\documentclass{article}
\usepackage{pinlabel}
\usepackage{graphicx}
\usepackage{amsmath}
\usepackage{amssymb}
\usepackage{amsthm}
\usepackage{multirow}
\usepackage{color}
\usepackage{soul}  
\usepackage{hyperref}
\usepackage[english]{babel}
\usepackage{cancel}
\usepackage[dvipsnames]{xcolor}
\usepackage{rotating}

\newcommand{\pic}[2]{\raisebox{-.5\height}
{\includegraphics[scale=#2]{#1}}}

\def\DiagramaPretzel{\pic{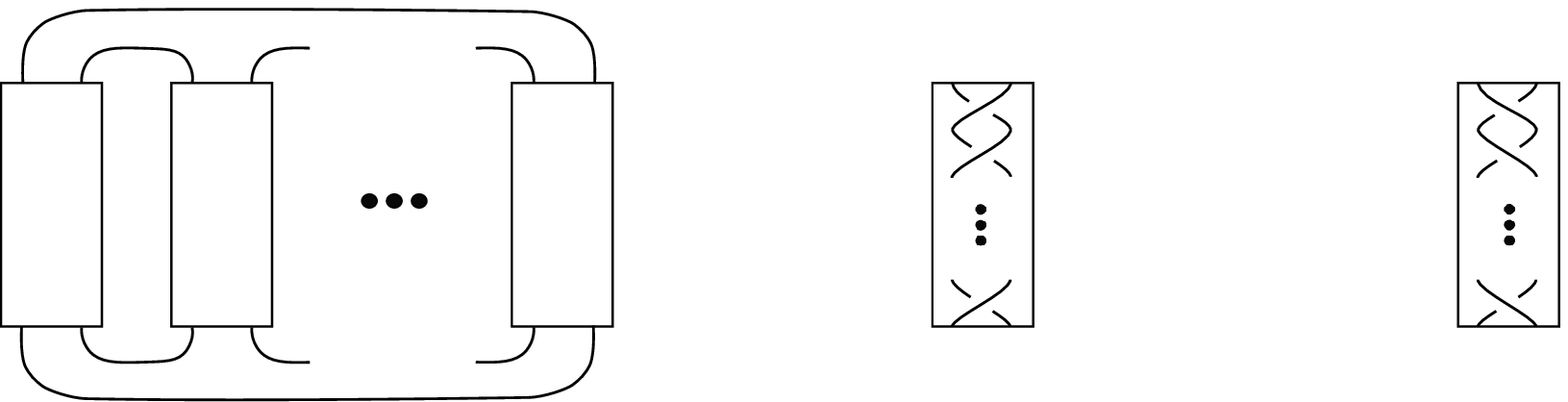}{.500}}
\def\MovimientoReduccion{\pic{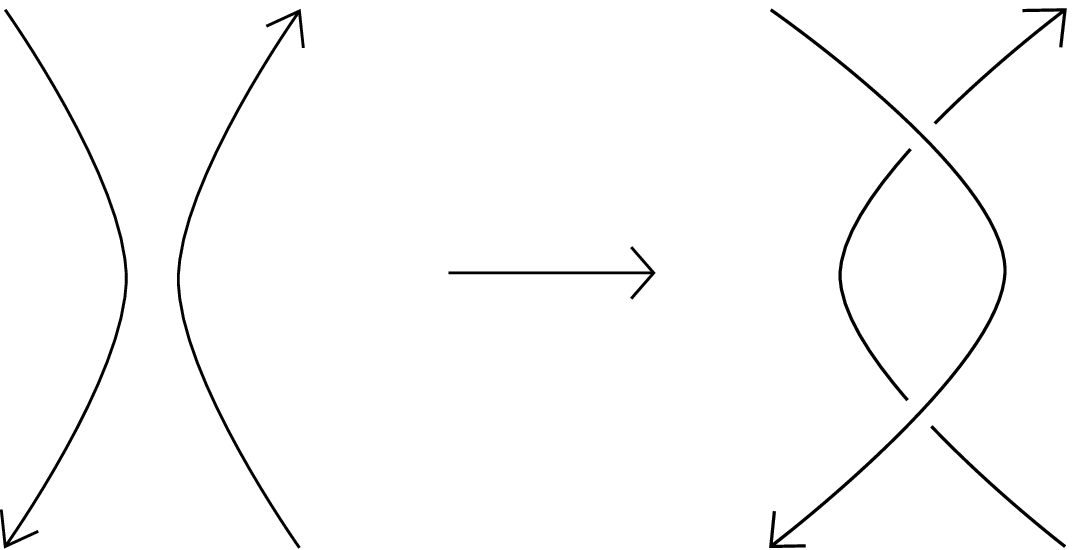}{.300}}
\def\MovimientoBasicoAislado{\pic{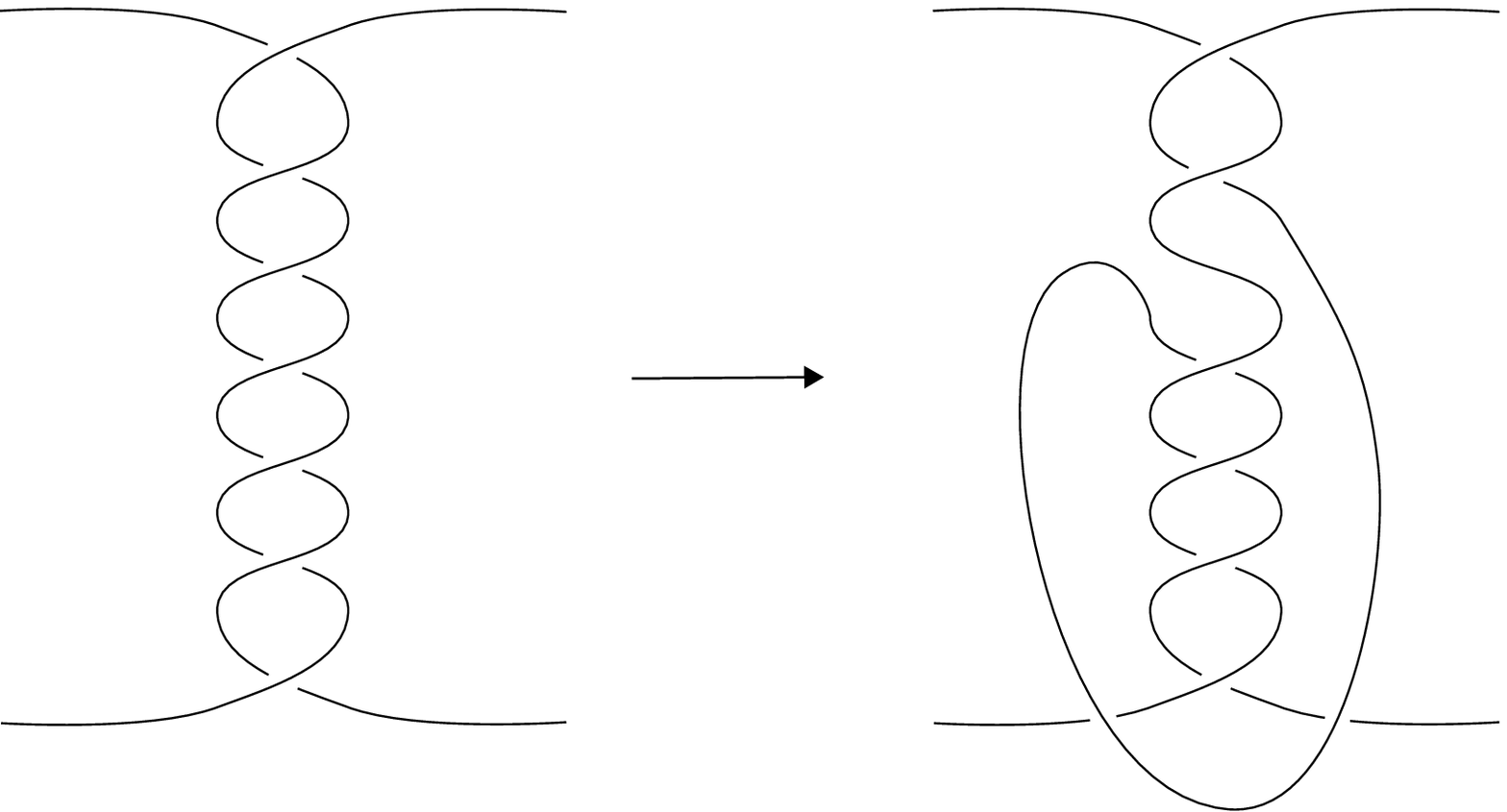}{.290}}
\def\MovimientoBasicoNumeroImparDeCruces{\pic{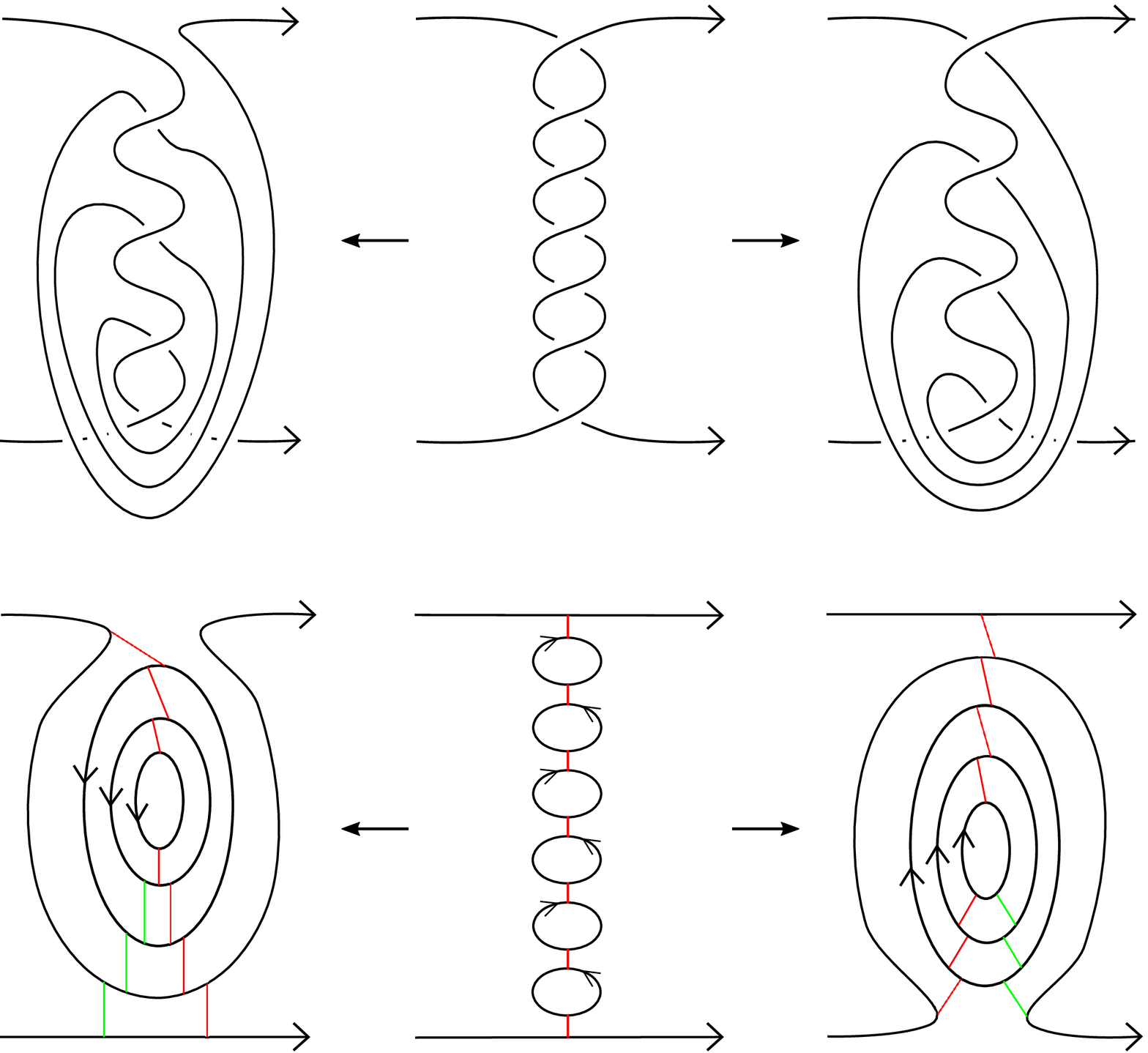}{.280}}
\def\MovimientoBasicoNumeroParDeCruces{\pic{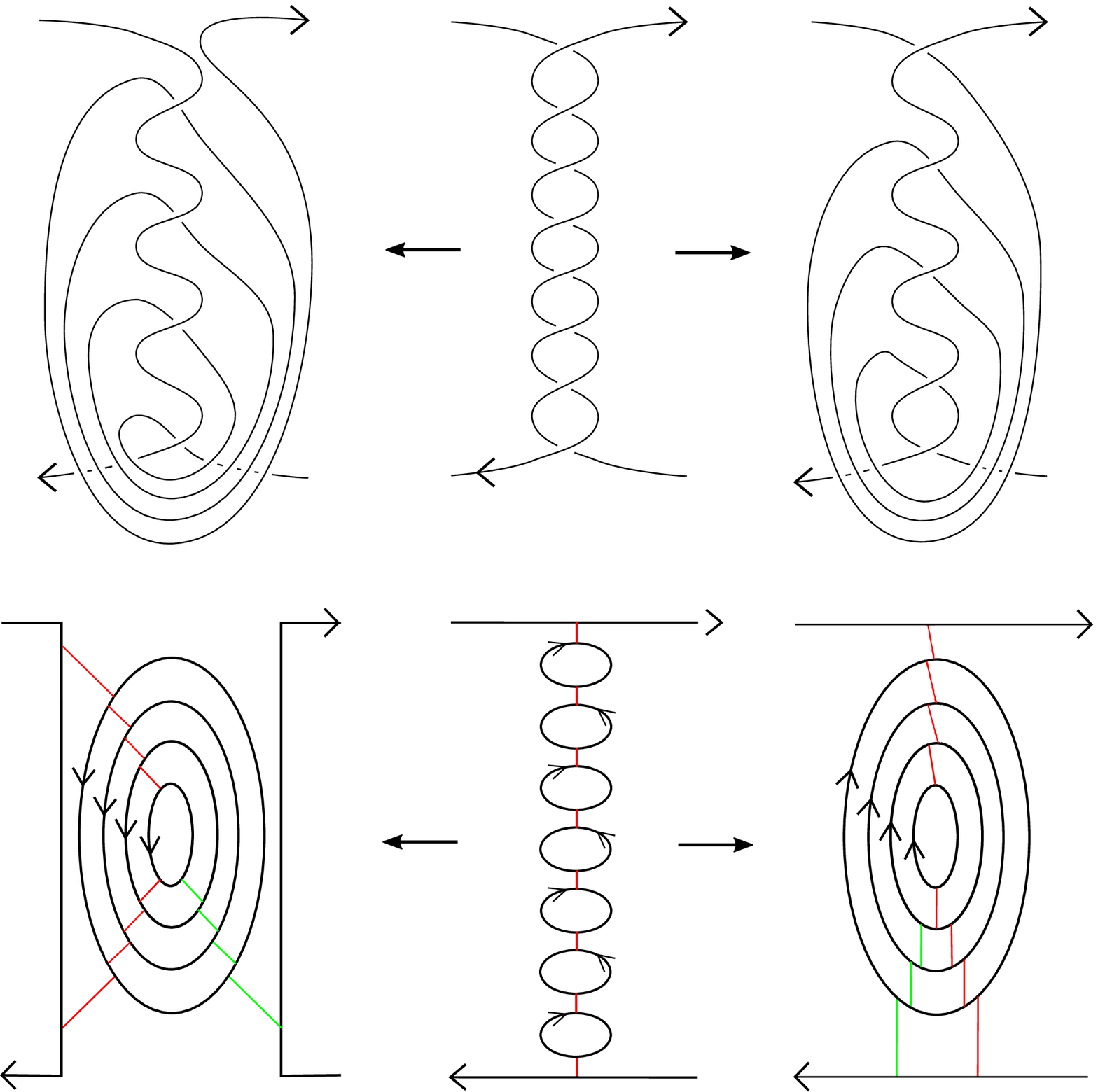}{.280}}
\def\DesplazamientoColumna{\pic{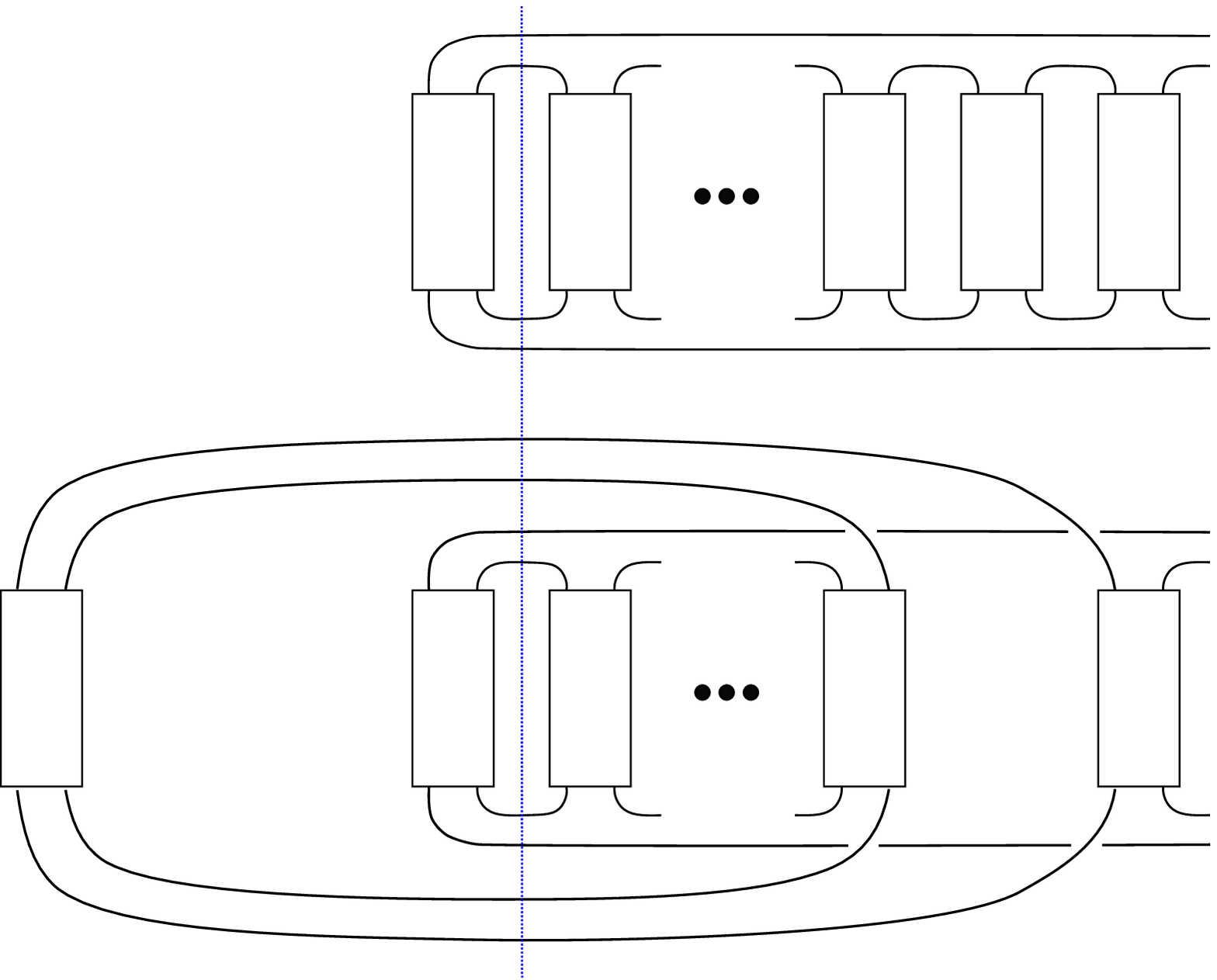}{.600}}
\def\NTresNoTodasImparesOrientacion{\pic{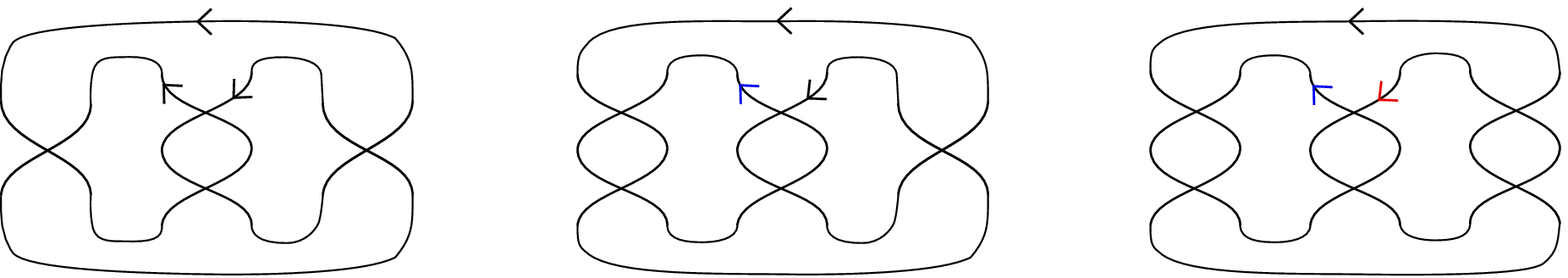}{.400}}
\def\NTresNoTodasImparesTrenza{\pic{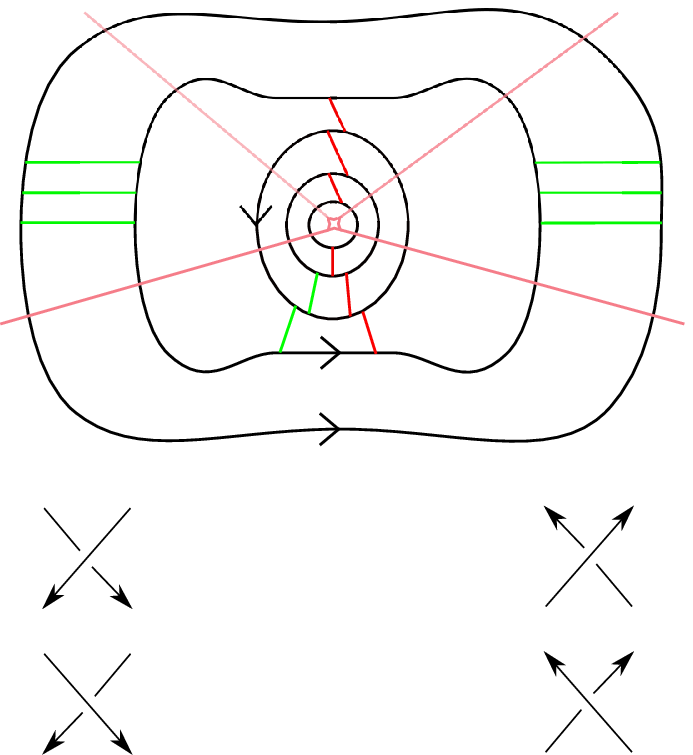}{.400}}
\def\NTresTodasImparesOrientacion{\pic{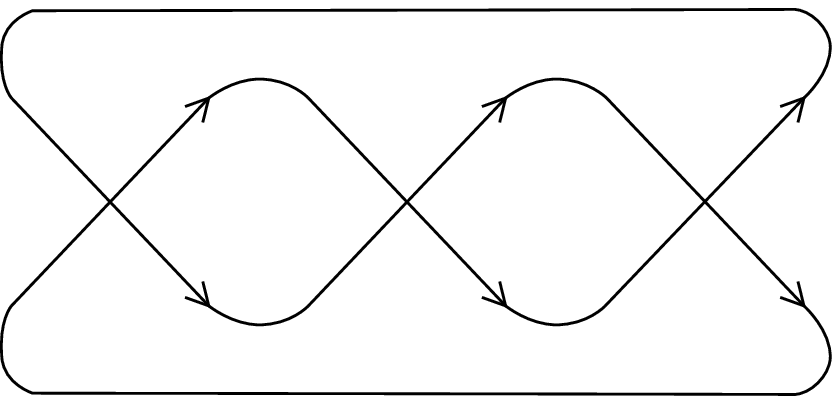}{.370}}
\def\NTresTodasImparesMovimientosBasicos{\pic{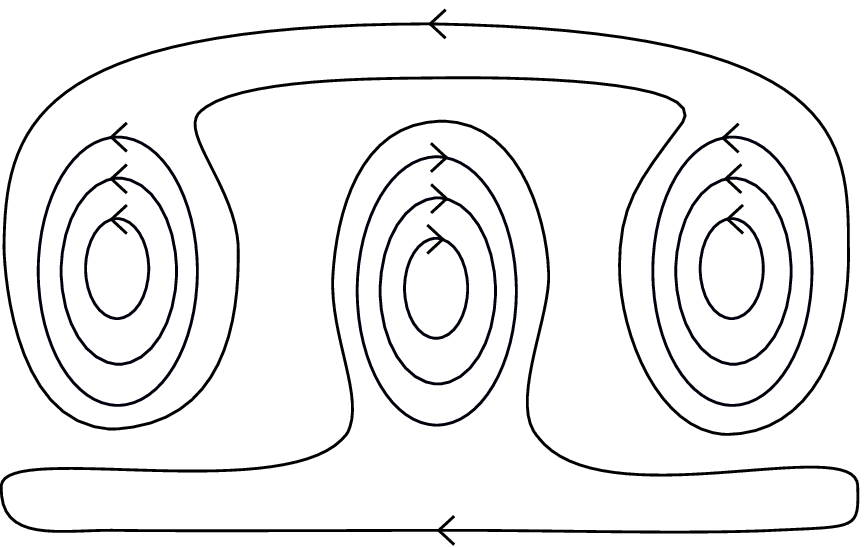}{.370}}
\def\NparOrientacion{\pic{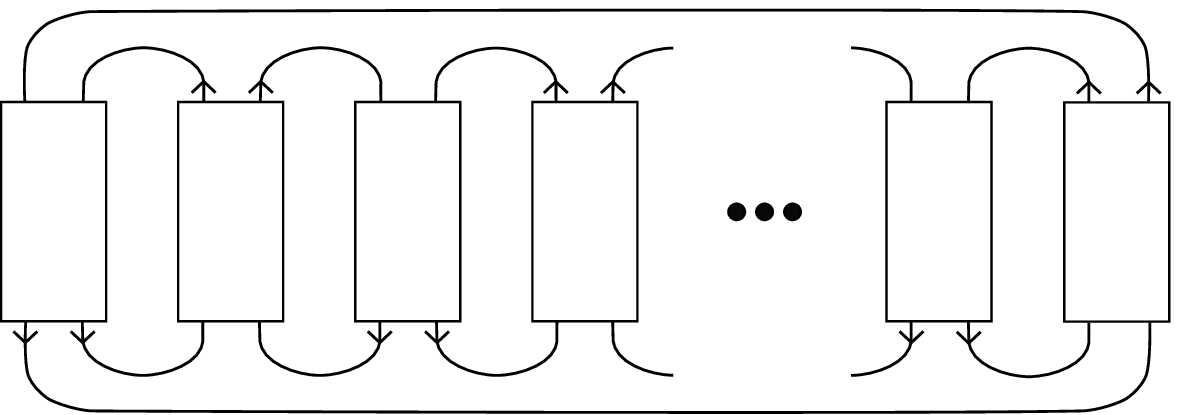}{.485}}
\def\NparCirculosSeifert{\pic{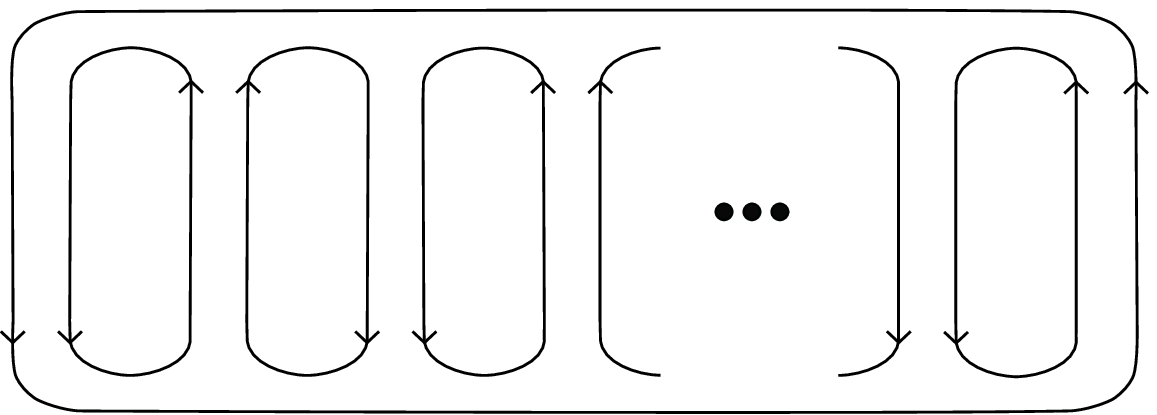}{.485}}
\def\NparDesplazamientos{\pic{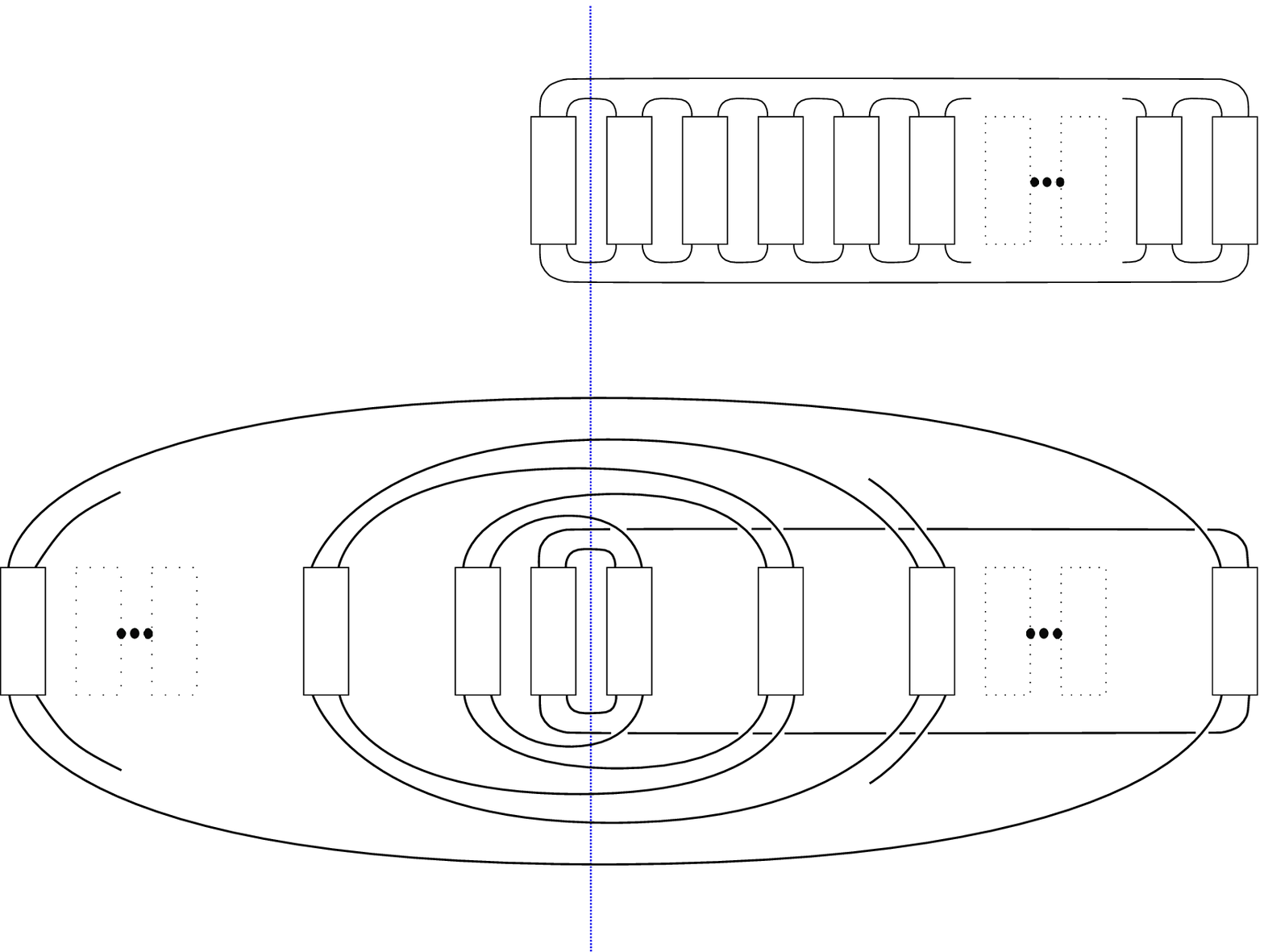}{.550}}
\def\NparTrenza{\pic{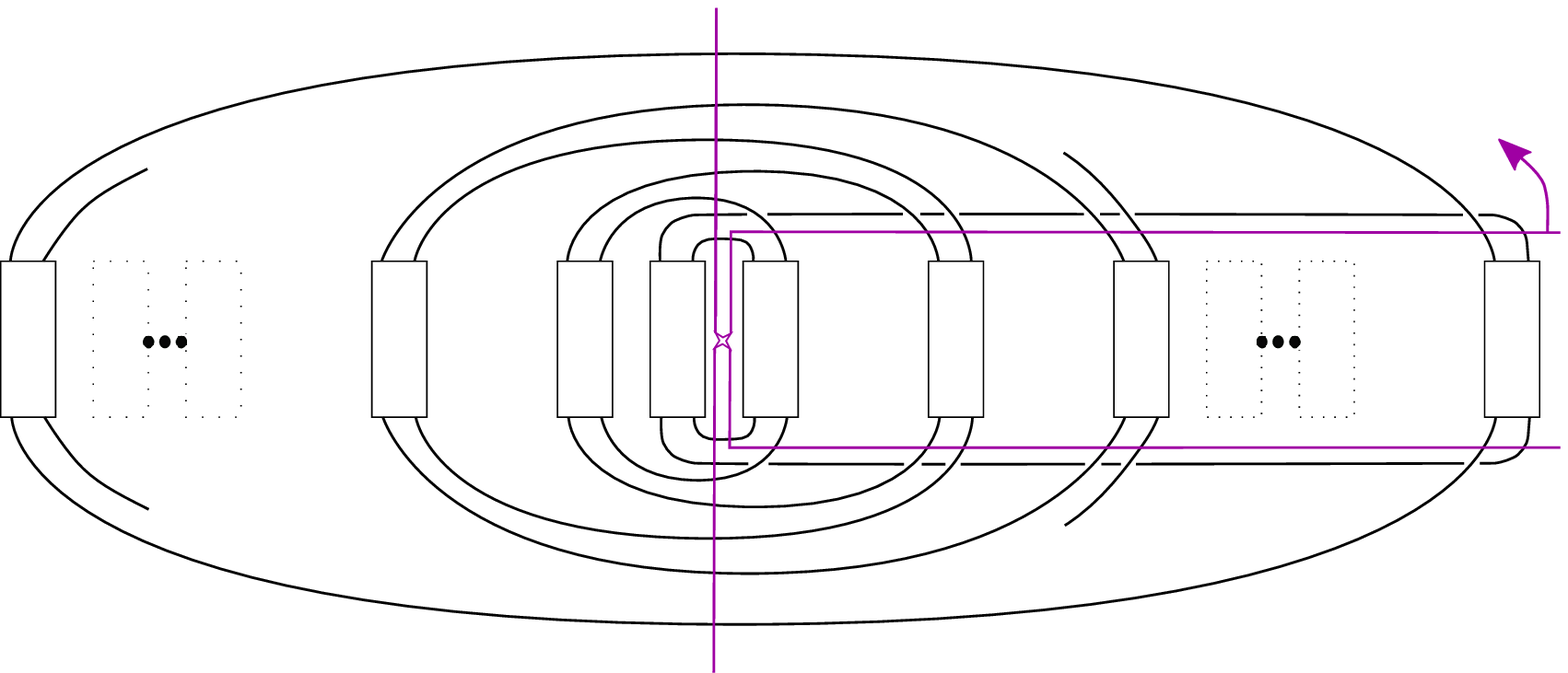}{.550}}
\def\NimparPrimeraColumnaParOrientacion{\pic{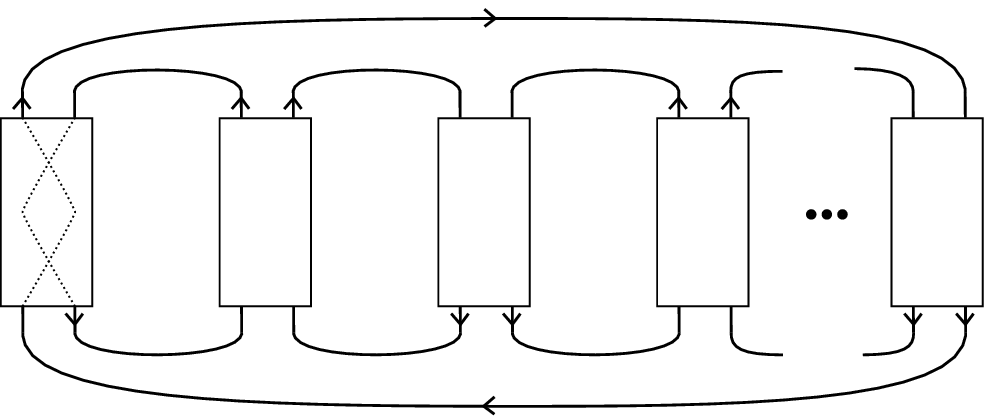}{.500}}
\def\NimparPrimeraColumnaParOrientando{\pic{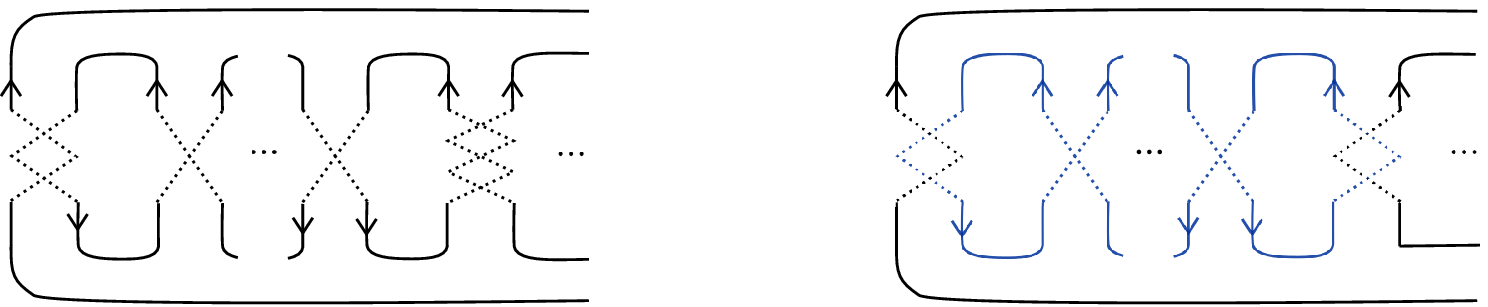}{.500}}
\def\NimparPrimeraColumnaParMovimientosBasicos{\pic{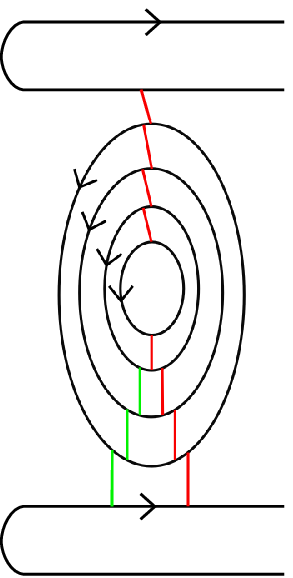}{.500}}
\def\NimparPrimeraColumnaParDesplazamientos{\pic{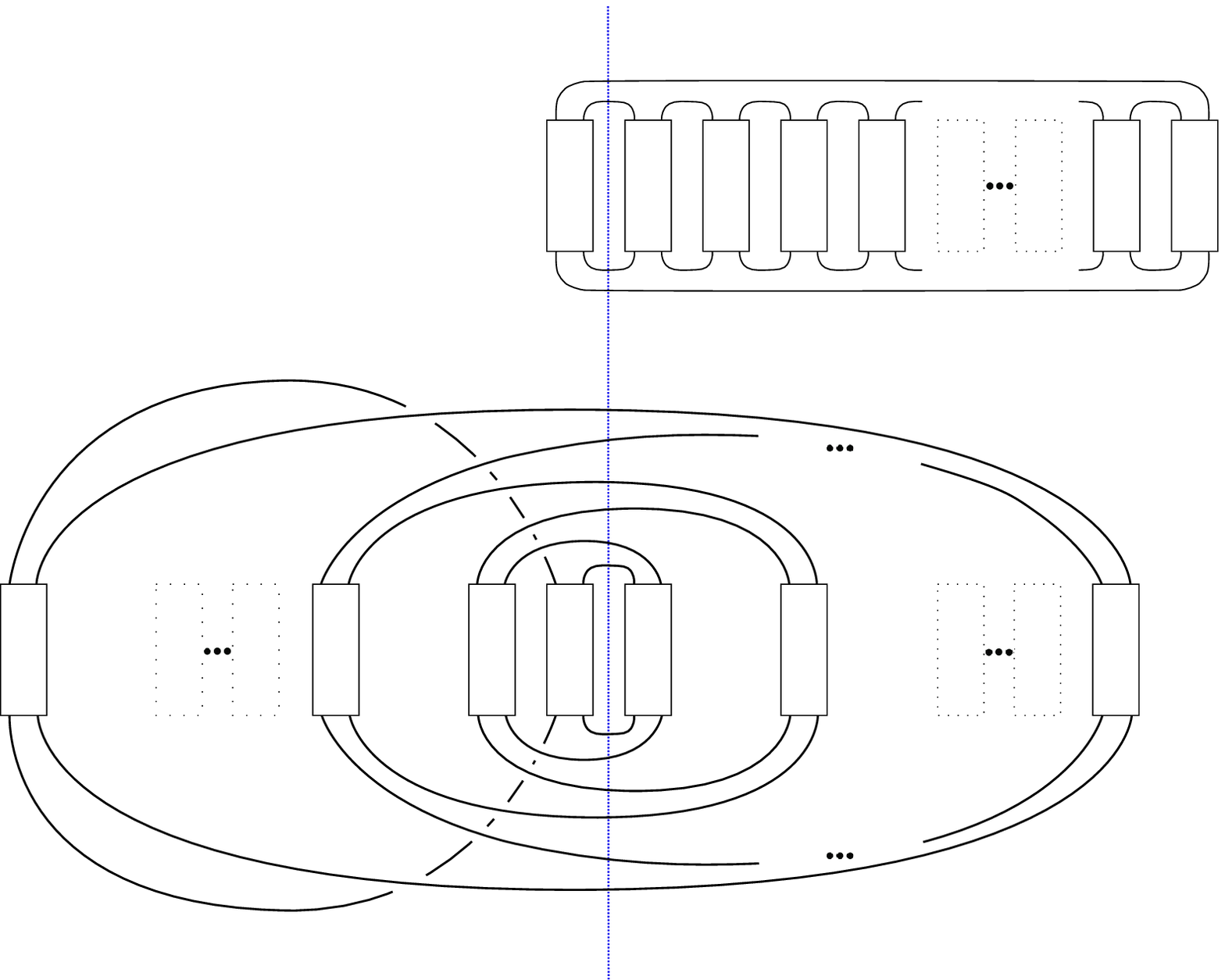}{.650}}
\def\NimparPrimeraColumnaParTrenza{\pic{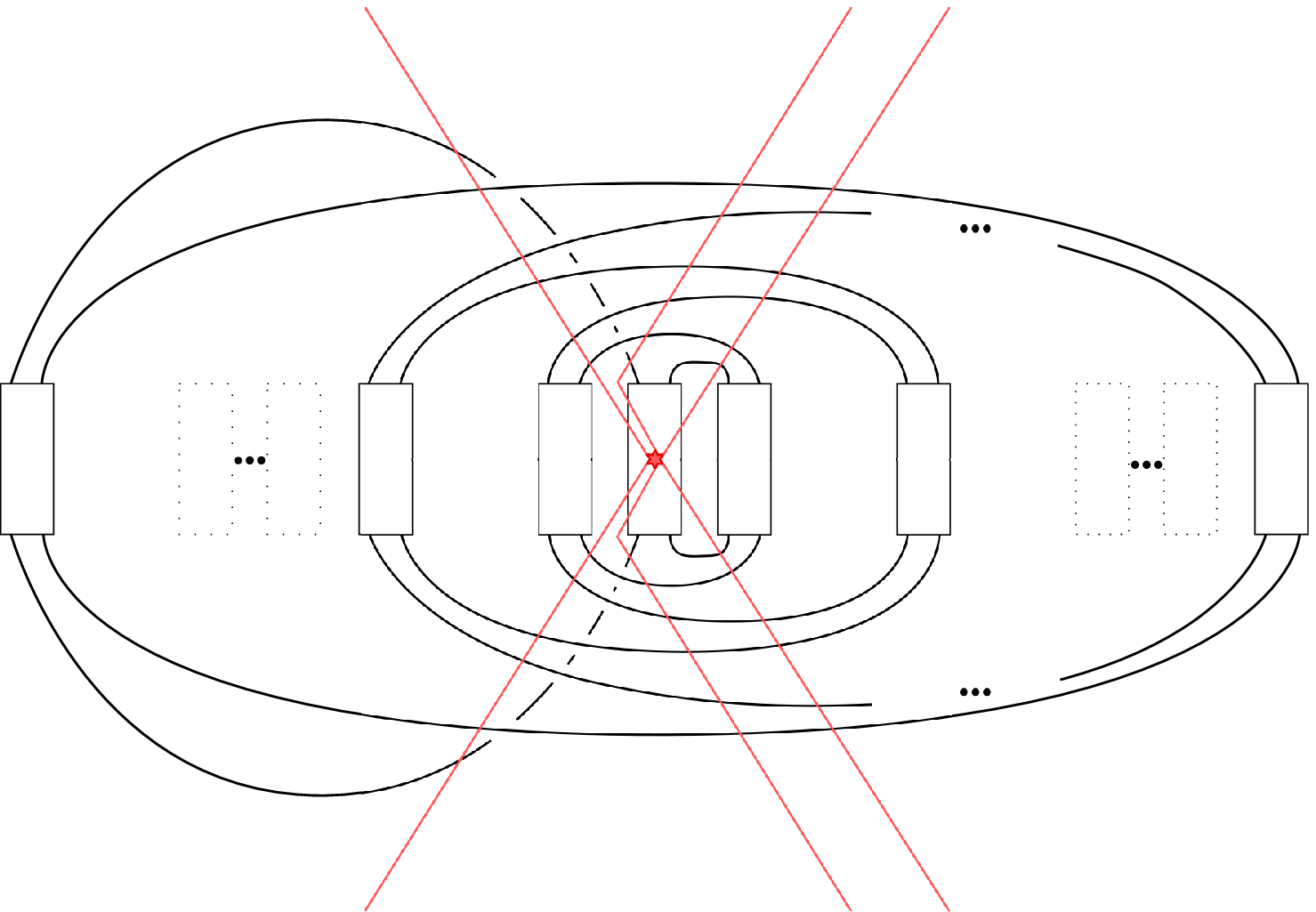}{.650}}
\def\NimparTodaEntradaImparOrientacion{\pic{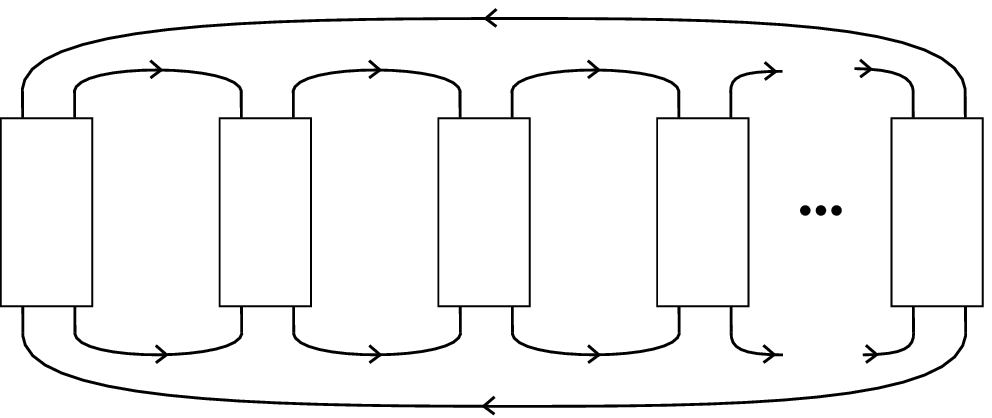}{.500}}
\def\NimparTodaEntradaImparMovimientosBasicos{\pic{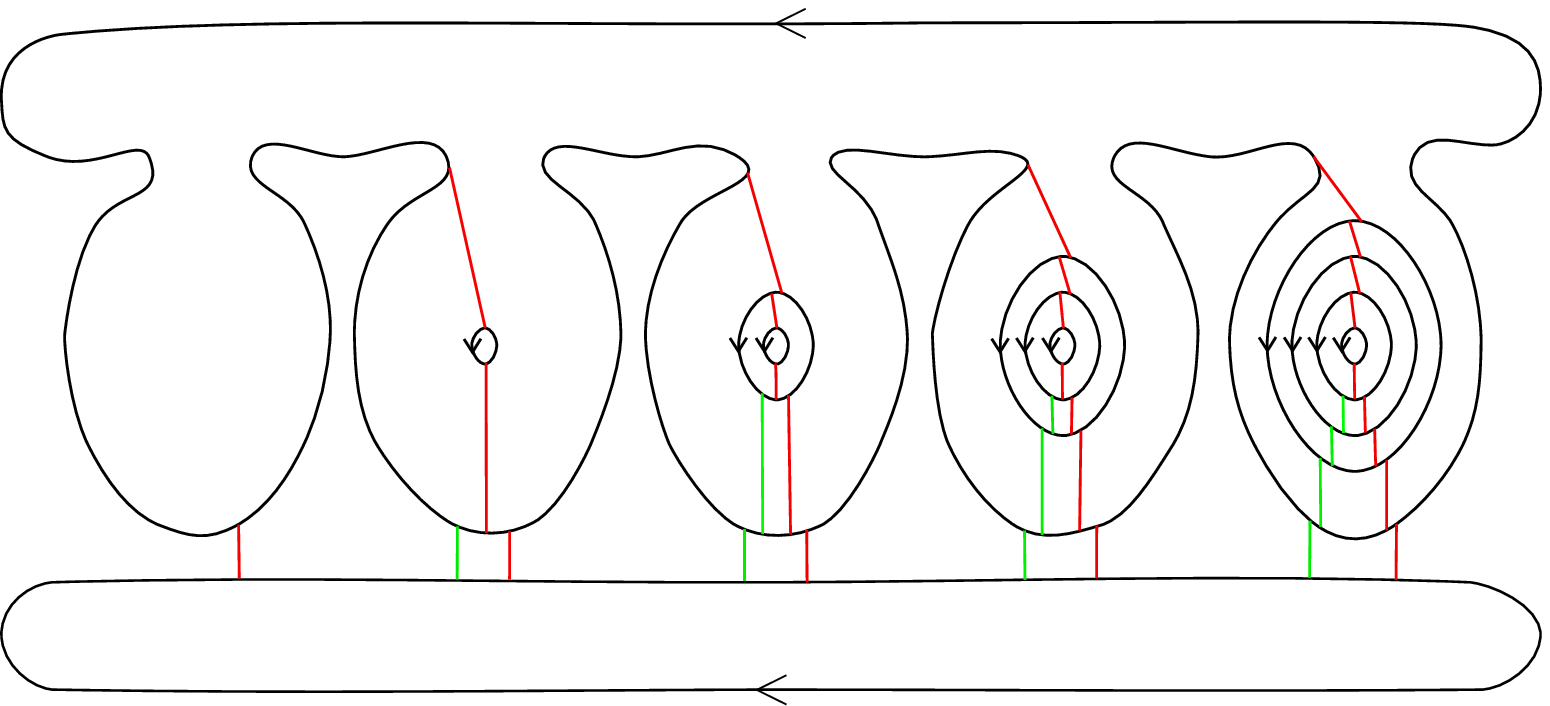}{.450}}
\def\NimparTodaEntradaImparEspacioParaReduccion{\pic{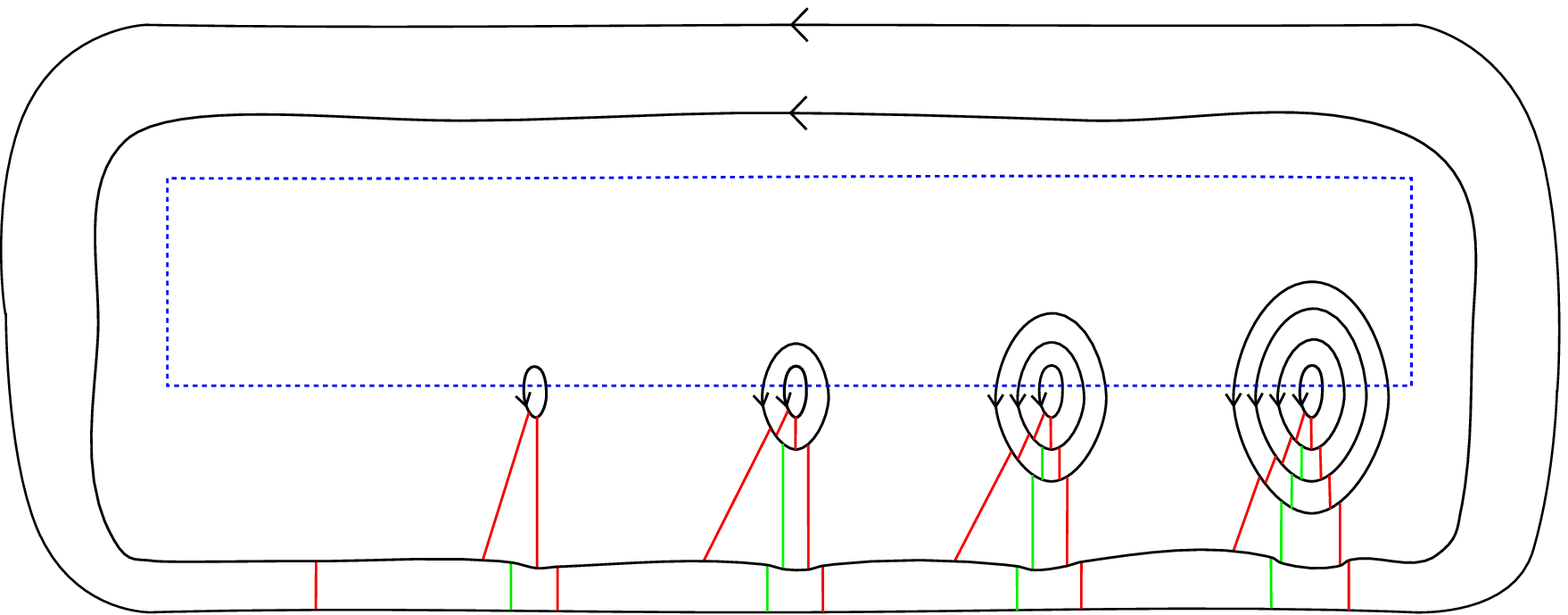}{.450}}
\def\NimparTodaEntradaImparReduccionMultiple{\pic{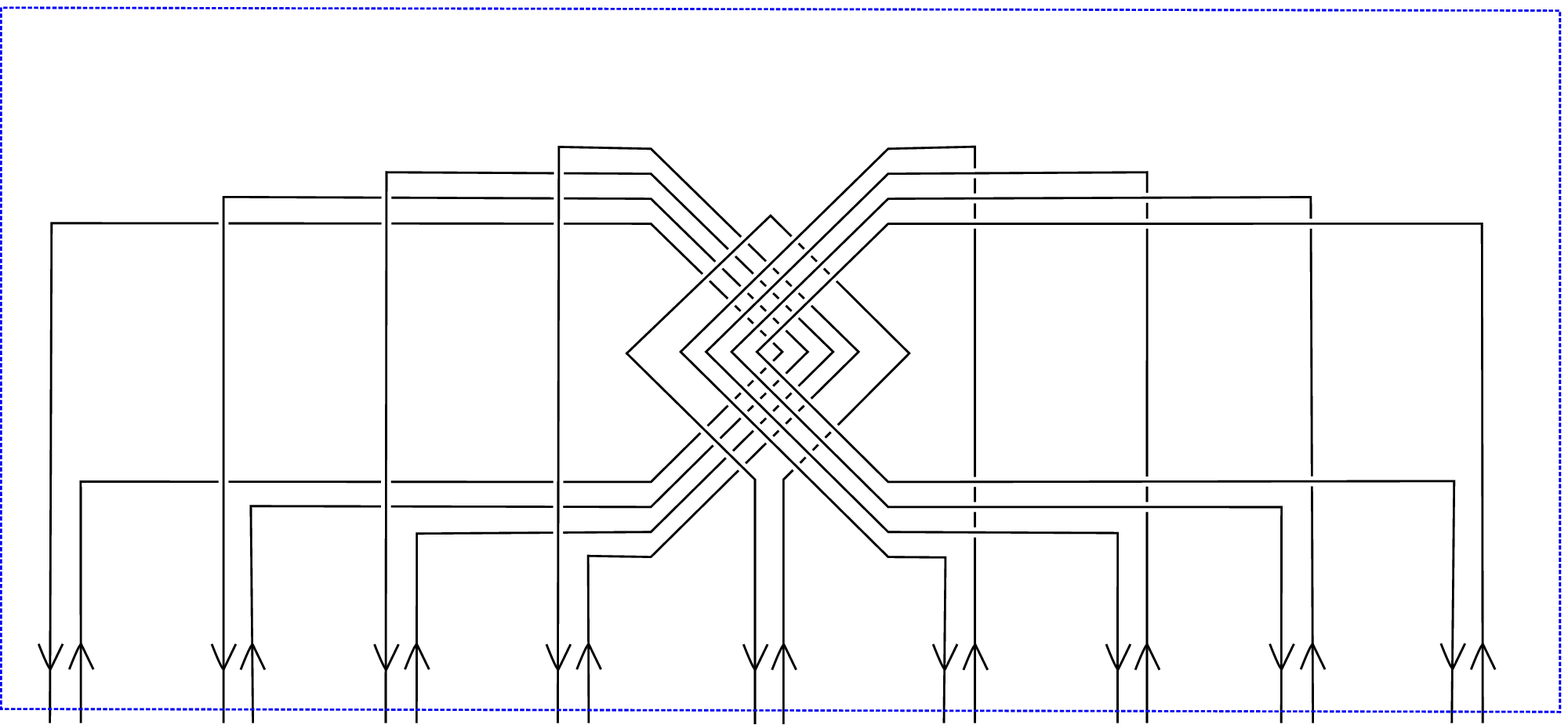}{.650}}
\def\NimparTodaEntradaImparReduccionMultipleDetalle{\pic{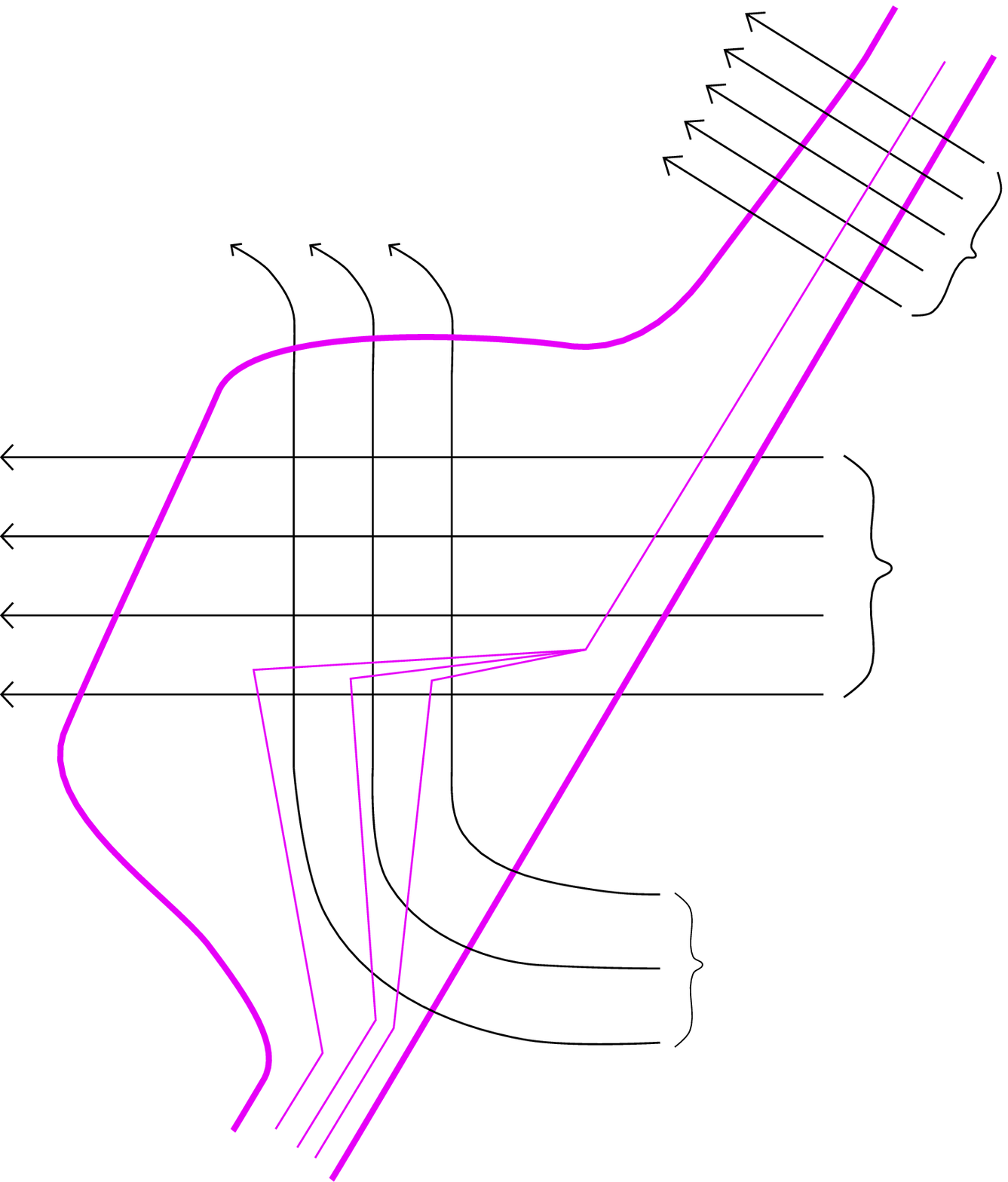}{.500}}
\def\NimparTodaEntradaImparGrupos{\pic{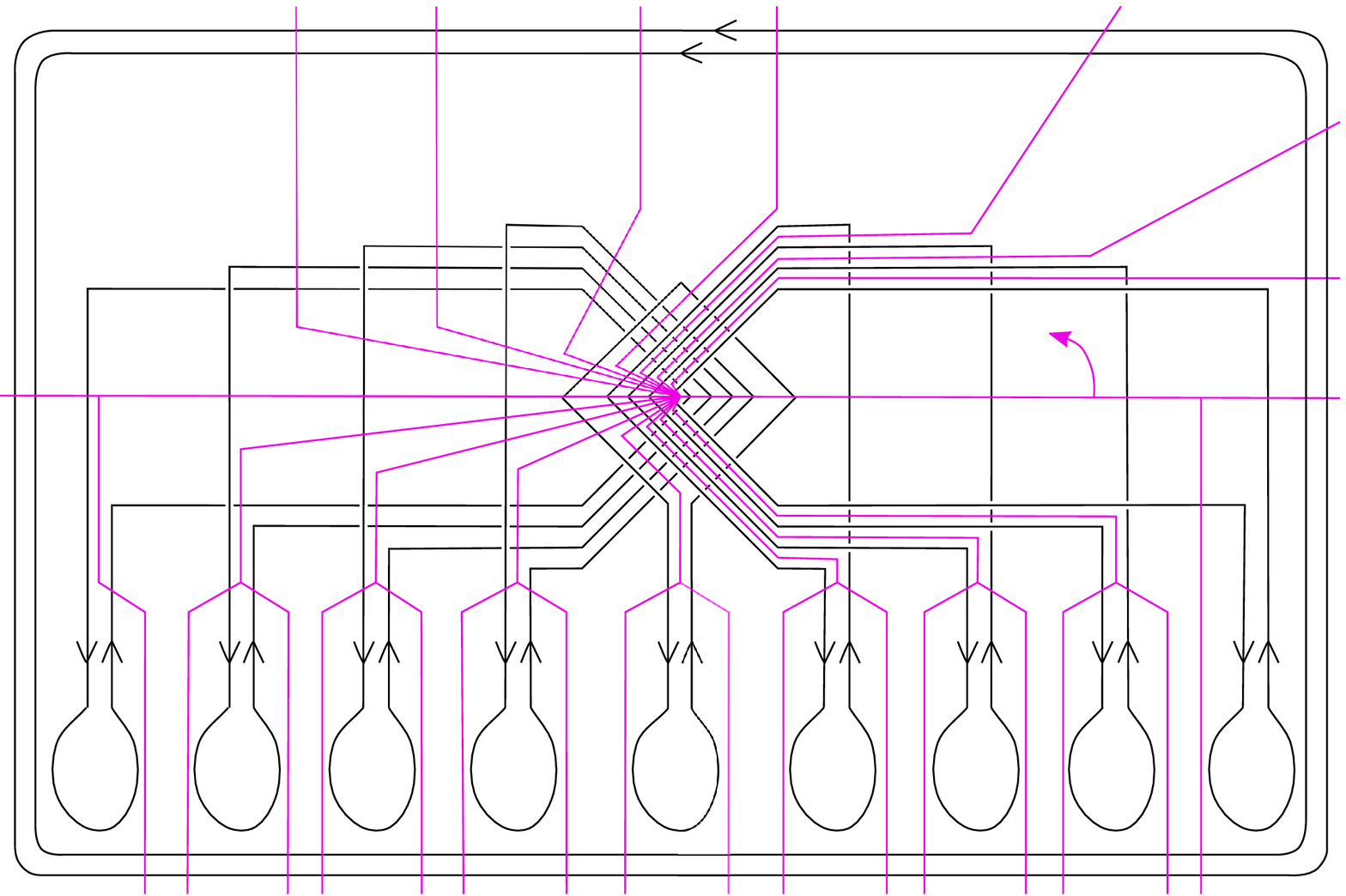}{.650}}
\def\Ejemplo{\pic{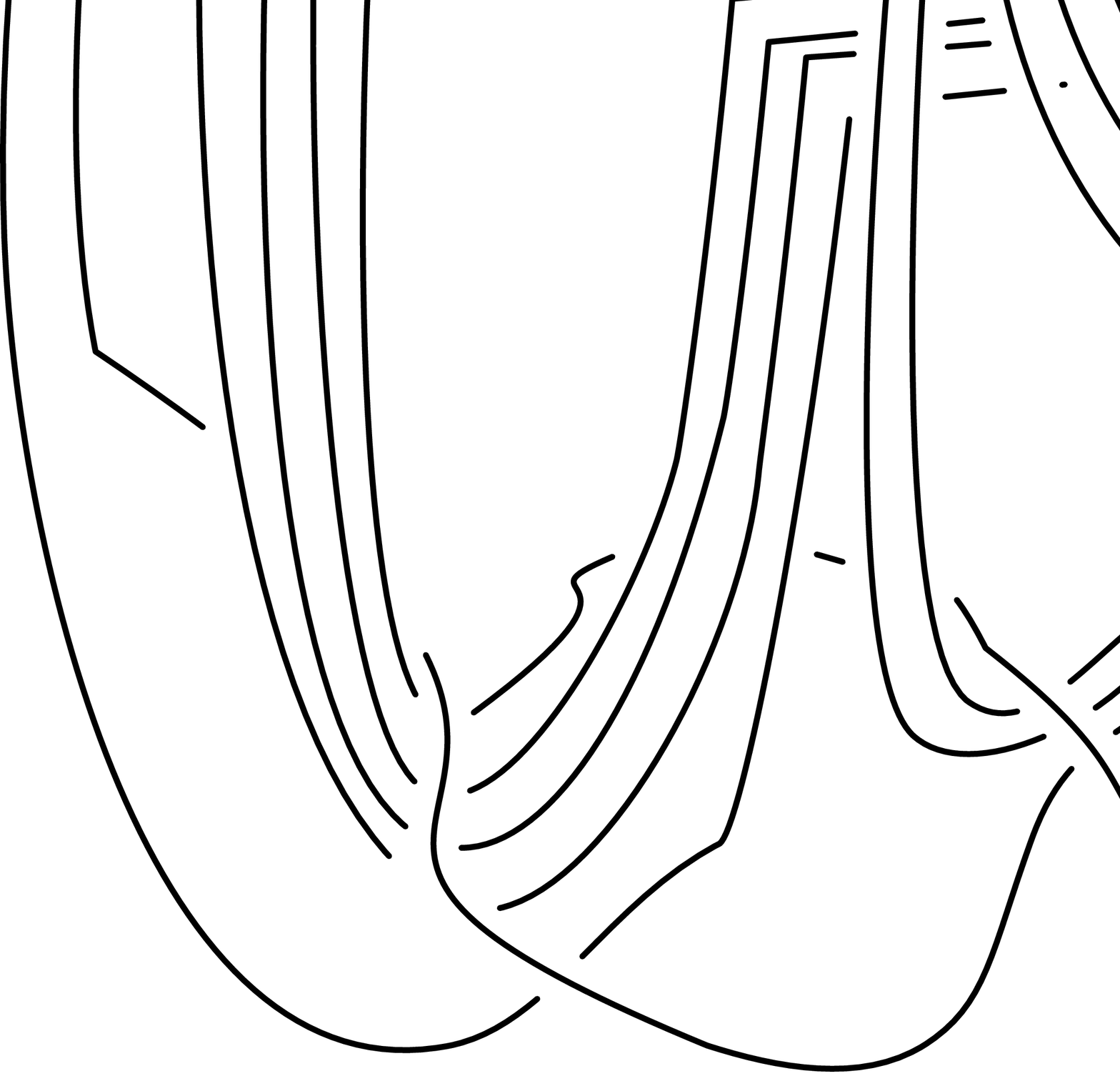}{.080}}

\newtheorem{theorem}{Theorem}

\renewenvironment{proof}[1][Proof]{\textit{#1.} }
{\hfill \rule{0.5em}{0.5em}}

\newcommand{\tn}{\textnormal}
\newcommand{\da}{\downarrow}
\newcommand{\ua}{\uparrow}

\begin{document}







\title{Braids for pretzel links}

\author{A. Del Pozo Manglano and P. M. G. Manch\'on \footnote{\hspace{-0.02cm}The second author is partially supported by MEC-FEDER grant MTM2016-76453-C2-1-P.}}
\maketitle

\begin{abstract}
We give a general procedure that provides, given any particular pretzel link, a braid whose closure is the pretzel link. Moreover, we manage to give a specific braid word in terms of the entries of the pretzel link.
\end{abstract}


\textbf{MSC Class:} \emph{57M25.}

\section{Introduction} \label{SectionIntroduction}
Given integers $a_1,...,a_n$, denote by $P(a_{1},...,a_{n})$ the pretzel link diagram shown in Figure \ref{FigureDiagramaPretzel}. Here $a_i$ indicates $|a_i|$ crossings, with signs $a_i/|a_i|$ if $a_i\neq 0$ (we follow the notation in \cite{Lickorish}).
\begin{figure}[ht!]
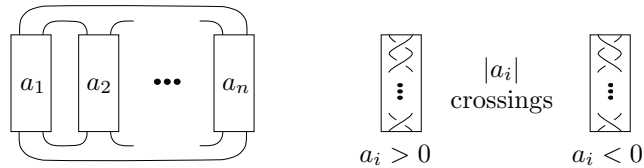
 
\labellist
 \pinlabel {$a_1$} at 17 57
 \pinlabel {$a_2$} at 67 57
 \pinlabel {$a_n$} at 170 57
 \pinlabel {$a_i>0$} at 290 5
 \pinlabel {$a_i<0$} at 450 5
 \pinlabel {\begin{tabular}{c}$|a_i|$\\ crossings\end{tabular}} at 370 57
\endlabellist
\begin{center}
\DiagramaPretzel
\end{center}
\caption{Pretzel link diagram $P(a_{1},...,a_{n})$.}
\label{FigureDiagramaPretzel} 
\end{figure}

A pretzel link is a link that has a pretzel diagram. It is not difficult to see that $P$ represents a link with more than one component if there are two or more even entries, or if there is an even number of entries and all of them are odd ($n$ is even and each $a_i$ is odd).

There are many algorithms for calculating polynomial invariants for a link, from a braid whose closure is the link (see for example the programs by Morton and Short \cite{Morton}). The necessity of computing polynomial invariants for pretzel links was the original motivation for trying to find, from any pretzel link $P=P(a_1, \dots, a_n)$, a braid $\beta$ whose closure is $P$, i.e., $P=\hat{\beta}$. Some particular cases of these braids can be found in \cite{KnotAtlas} and \cite{KnotInfo}. The paper \cite{intento} addresses this issue but solves just a few examples.

Alexander \cite{Alexander} proved in 1923 that every link is the closure of a braid. However, his strategy is geometric and does not provide a practical algorithm for finding the braid. A more combinatorial (and recent) argument can be deduced from the work by Shuji Yamada \cite{Yamada}, Pierre Vogel \cite{Vogel} and  Pawel Traczyk \cite{Traczyk}. Peter Cromwell nicely explains in his book \cite{Peter} the proof given by Pawel Traczyk \cite{Traczyk}, based on the so called reducing move shown in Figure~\ref{FigureMovimientoReduccion}. 
\begin{figure}[ht!]
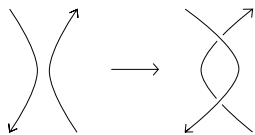

\begin{center}
\MovimientoReduccion
\end{center}
\caption{A reducing move. The two arcs on the left must belong to different Seifert circles.} 
\label{FigureMovimientoReduccion}
\end{figure}

In general, any diagram can be converted in a braid diagram performing some reducing moves. The complexity of an oriented diagram is the number of incompatible pairs of oriented Seifert circles, where two oriented disjoint circles in the sphere are said to be incompatible if both circles have the orientation inherited from an orientation in the annulus which is a cobordism of them (a less sophisticated definition is the following: two oriented Seifert circles are incompatible if, seen as gears that rotate according to their orientations, are incompatible). A diagram is a braid diagram if and only if its complexity is zero. Then one can check that each reducing move reduces the complexity of the diagram by one, and that if the complexity is greater than zero, then it is possible to apply a reducing move. 

However, the direct application of this strategy over the pretzel diagrams drives us quickly to an intricate mess. This has obligated us to develop a couple of new moves on a pretzel diagram, identified in the paper as {\it basic move on a column} and {\it column shift}. These two moves, together with a multiple reducing move, will be the tools for obtaining our algorithms. 

The obtained braids are far from being minimal respect to the number of crossings. However, we have checked in many cases that they are minimal respect to the number of strands. Theorem~\ref{TheoremEvenEntries} provides a braid with $4$ strands for the pretzel knot $P(1,1,1,-2)$ although its minimal braid number is $3$; nevertheless one should note that this pretzel knot can be written with only three entries.

The paper is organized as follows: in Section~\ref{SectionMoves} we introduce the two special types of moves, to be applied on pretzel diagrams. Section~\ref{SectionResults} contains the general strategy to go from any arbitrary pretzel diagram to the braid diagram, distinguishing whether there is an even number of entries (Theorem~\ref{TheoremEvenEntries}) or an odd number of entries (Theorem~\ref{TheoremOddEntries}). In both cases specific braid words for each pretzel diagram are displayed; these words are in some cases a bit complicated, and needs some codification that we provide with the necessary notation. The case of three entries is shown apart for its special relevance (Theorem~\ref{TheoremThreeEntries}).

\section{Special moves} \label{SectionMoves}
We first introduce the special moves that will be applied on the pretzel diagram in order to obtain the wanted braid. 
A basic move (in a column of crossings) takes the over (under) segment of a crossing of a column and take it over (under) the rest of the column. Figure~\ref{FigureMovimientoBasicoAislado} shows a basic move of the third over segment. Note that a basic move increases the number of crossings by one except if it is applied on the last crossing, in which case it has no effect at all.
\begin{figure}[ht!]
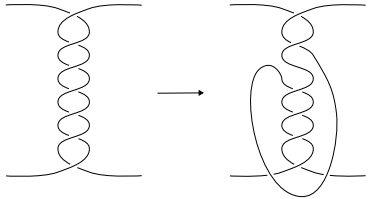

\begin{center}
\MovimientoBasicoAislado
\end{center}
\caption{A basic move.} \label{FigureMovimientoBasicoAislado}
\end{figure}

If a column has its two strands oriented in opposite directions, applying appropriate basic moves gives a nice configuration of concentric Seifert circles. 

Precisely, if the number $|a|$ of crossings of the column is odd, performing $\frac{|a|-1}{2}$ appropriate basic moves we obtain exactly $\frac{|a|-1}{2}$ concentric Seifert circles. The basic moves can be performed on the crossings in odd positions (first, third and so on until the antepenultimate), or in even positions (second, fourth, etc.), as Figure~\ref{FigureMovimientoBasicoNumeroImparDeCruces} shows in the case $a=7$. In a Seifert circles diagram, a colored segment (we call it a scar) recalls us where a crossing was: it is colored green if the crossing was positive, red if negative. 
\begin{figure}[ht!]
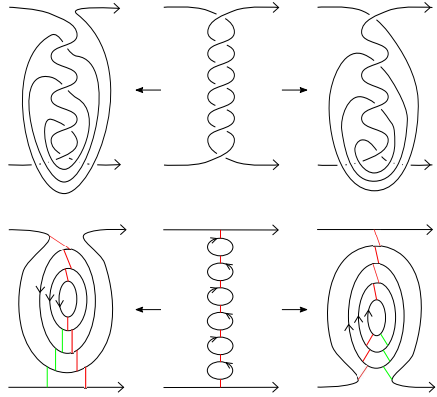

\begin{center}
\MovimientoBasicoNumeroImparDeCruces
\end{center}
\caption{Basic moves on a column with an odd number $|a|$ of crossings  (here $a=7$). The top central picture is the original column of crossings. Basic moves on the odd (even) crossings produce the picture on the left (right). How it affects the Seifert circles is shown in the bottom line.} 
\label{FigureMovimientoBasicoNumeroImparDeCruces}
\end{figure}

If $|a|$ is even, performing $\frac{|a|}{2}$ appropriate basic moves we obtain exactly $\frac{|a|}{2}$ concentric Seifert circles. Once more, the basic moves can be performed on the crossings in odd positions obtaining $\frac{|a|}{2}$ extra crossings, or in even positions obtaining $\frac{|a|}{2}-1$ extra crossings (the last basic move has no effect at all), as Figure~\ref{FigureMovimientoBasicoNumeroParDeCruces} shows in the case $a=8$.
\begin{figure}[ht!]
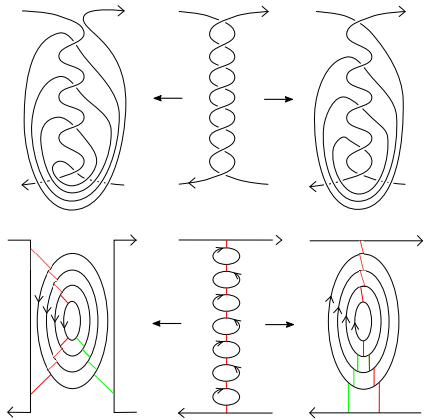

\begin{center}
\MovimientoBasicoNumeroParDeCruces
\end{center}
\caption{Basic moves on a column with an even number $|a|$ of crossings (here $a=8$). The top central picture is the original column of crossings. Basic moves on the odd (even) crossings produce the picture on the left (right). How it affects the Seifert circles is shown in the bottom line.}
\label{FigureMovimientoBasicoNumeroParDeCruces}
\end{figure}

The following remarks are valid whether $|a|$ is even or odd (Figures~\ref{FigureMovimientoBasicoNumeroImparDeCruces} and \ref{FigureMovimientoBasicoNumeroParDeCruces}). First, note that if we reverse the orientation of the two strands, then the orientation of the Seifert circles must be reversed in both initial and final Seifert circle diagrams. Second, in both Figures~\ref{FigureMovimientoBasicoNumeroImparDeCruces} and \ref{FigureMovimientoBasicoNumeroParDeCruces} we show cases in which $a>0$; if $a<0$ we must then isotopy the under segments: the oriented Seifert circles remain the same, but the scars must interchange their color in both initial and final Seifert circle diagrams.

The basic move can be seen as an intra-column move. We will now describe an inter-column version of it. A {\it column shift} is shown in Figure~\ref{FigureDesplazamientoColumna}. The chosen column is moved out to the left by a $\pi$-rotation over the paper around a vertical axis situated between the first and second columns. This creates four new crossings, two with the top horizontal line, two with the bottom horizontal line. Note that out of the pretzel diagram we obtain an exact copy of the shifted column, keeping the same {\it unoriented} sign of the crossings of the original column.
\begin{figure}[ht!]
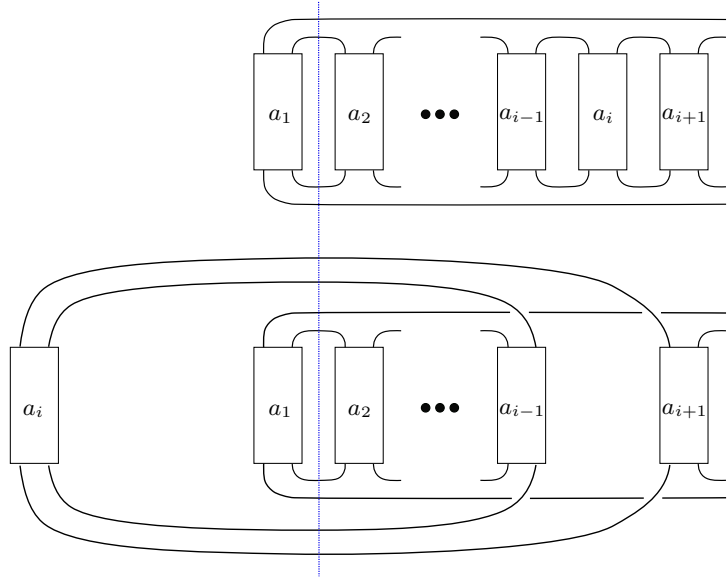

\labellist
\small
 \pinlabel {$a_1$} at 170 290
 \pinlabel {$a_2$} at 220 290
 \pinlabel {$a_{i-1}$} at 322 290
 \pinlabel {$a_i$} at 373 290
 \pinlabel {$a_{i+1}$} at 424 290
 \pinlabel {$a_1$} at 170 105
 \pinlabel {$a_2$} at 220 105
 \pinlabel {$a_{i-1}$} at 322 105
 \pinlabel {$a_i$} at 15 105
 \pinlabel {$a_{i+1}$} at 424 105
 \pinlabel {\begin{tabular}{c}$|a_i|$\\ crossings\end{tabular}} at 1285 175
\endlabellist
\begin{center}
\DesplazamientoColumna
\end{center}
\caption{Shift of the $i$th-column.} \label{FigureDesplazamientoColumna}
\end{figure}

Finally, in this Section we introduce some notation that will make easier the reading of the braid words that will appear later. 
As usual, the sequence of non-zero integers $(m_1, \dots, m_k)$ codes the braid word $\sigma_{|m_1|}^{\epsilon_1}\dots\sigma_{|m_k|}^{\epsilon_k}$
where $\epsilon_i=|m_i|/m_i$ is the sign of the integer $m_i$. Of course, $(m_1, \dots, m_r)(m_1', \dots, m_s') = (m_1, \dots, m_r, m_1', \dots, m_s')$. It will also be convenient to have a specific notation for some types of sequences of non-zero integers. If $i$ and $j$ are integers, $[i \ua j]$ will denote the sequence of integers $(i, i+1, \dots, j-1, j)$ if $i\leq j$ and $ij>0$; otherwise it will denote the empty braid word. For example, $[-5 \ua -3] = (-5,-4,-3) = \sigma_5^{-1} \sigma_4^{-1}\sigma_3^{-1}$ but $[-3 \ua -5] = \emptyset$ or $[-3\ua 5]=\emptyset$. Analogously it is considered the notation $[i \da j]$.
Also, if $a$ is a non-zero integer, we denote its sign $|a|/a$ by $s_a$ and assume the equality $s_a(m_1, \dots, m_k) = (s_am_1, \dots, s_am_k)$. Then $s_a[i \ua j]$ is $[i\ua j]$ if $a>0$ and $[-i \da  -j]$ if $a<0$. Analogously, $s_a[i \da j]$ is $[i\da j]$ if $a>0$ and $[-i \ua  -j]$ if $a<0$. Finally, $[i \ua j]^{\da m \da}$ where $m$ is a non-negative integer means $[i \ua j] [i-1 \ua j-1] [i-2 \ua j-2]\dots$ until we complete $m$ groups (in particular it is the empty braid word if $m=0$). Precisely, 
$$
[i \ua j]^{\da m \da}
= [i \ua j] [i-1 \ua j-1] \dots [i-(m-1) \ua j-(m-1)].
$$
Also reversing arrows we have, for example, $[2\ua4]^{\ua 3 \ua}  =(2,3,4,3,4,5,4,5,6)$ or $[-3\ua-1]^{\ua2\ua} =[-3\ua-1][-2\ua0]=(-3, -2, -1)$.

\section{Results} \label{SectionResults}
The pretzel diagram $P(a)$ corresponds to the trivial knot. Also, $P(a,b)=\hat{\beta}$ where $\beta$ is the braid with two strands $\beta = \sigma_1^{a+b}$. In the following theorem we will examine the case of three entries. Since we can present a pretzel diagram in a circular fashion \cite{Angel}, we can see its entries in a cyclic fashion. We will refer to this fact as the cyclic property of the pretzel diagrams. 

\begin{theorem} \label{TheoremThreeEntries} 
Let $P=P(a, b, c)$ be a pretzel diagram with three entries. Then $P=\hat{\beta}$ where:
\begin{enumerate}
\item If there is at least one even entry, that we may assume to be the central one $b$ by the cyclic property, then
$$
\beta = 1^{c} \, s_b[-2 \da -\frac{|b|}{2}-1] \, 1^a \, s_b\left([2 \ua \frac{|b|}{2}] \, [-\frac{|b|}{2}-1 \ua -2]\right).
$$

\item If there is no even entries at all, 
$$
\begin{array}{rcl}
\beta & = & [b'+c'+2 \da  b'+3]^{\ua a' \ua} \\
&& s_a([-b'-2 \da -b'-a'-1] [b'+1 \ua b'+a'] [-b'-a'-1 \ua -b'-1]) \\
&&[-b'-a'-2 \ua -b'-3]^{\da c'\da} \\
&&s_b([-b' \ua -1] [-1 \da -b'-1] [1 \ua b']) \\
&&s_c([-b'-2 \da -b'-c'-1] [b'+1 \ua b'+c'] [-b'-c'-1 \ua -b'-1])
\end{array}
$$
where
$$
a' = \frac{|a|-1}{2}, \quad
b' = \frac{|b|-1}{2} \quad \tn{ and } \quad
c' = \frac{|c|-1}{2}.
$$
\end{enumerate}

In the first case $\beta$ has $|a|+|c|+\frac{3|b|}{2}-1$ crossings and $\frac{|b|}{2}+2$ strands. In the second case $\beta$ has $|a|+|c|-1+\frac{3|b|+|ac|}{2}$ crossings and $\frac{|a|+|b|+|c|+1}{2}$ strands. 
\end{theorem}

\begin{proof}
Suppose that there is at least one even entry. By the cyclic property we may assume that the central entry is even, and orientation can be chosen in such a way that the strands of the central column run in opposite directions, and the top horizontal line goes to left, as shown in  Figure \ref{FigureNTresNoTodasImparesOrientacion}. 
\begin{figure}[ht!]
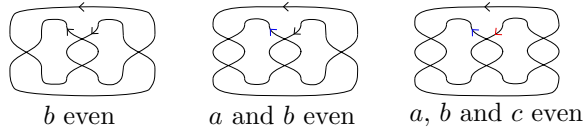
  
\labellist
 \pinlabel {$b$ even} at 67 -17
 \pinlabel {$a$ and $b$ even} at 263 -17
 \pinlabel {$a$, $b$ and $c$ even} at 470 -17
\endlabellist
\begin{center}
\NTresNoTodasImparesOrientacion 
\end{center}
\caption{Orientation for the pretzel $P(a,b,c)$ when at least one entry is even.} \label{FigureNTresNoTodasImparesOrientacion}
\end{figure}

Then we apply $\frac{|b|}{2}$ basic moves on the {\it even} crossings of the central column, adding $\frac{|b|}{2}-1$ crossings (since the last basic move is redundant). The result is (apply some isotopies if you want to see all the circles concentric in the plane, not just in the sphere) a set of $\frac{|b|}{2}+2$ Seifert circles with complexity zero (see Figure \ref{FigureNTresNoTodasImparesTrenza}). It is easy to check that in all these cases the braid word is as stated above; note that the crossings in the braid corresponding to the first entry have the same sign as $a$, and the same happens for the crossings in the third column of the pretzel diagram (see Figure \ref{FigureNTresNoTodasImparesTrenza}). 
\begin{figure}[ht!]
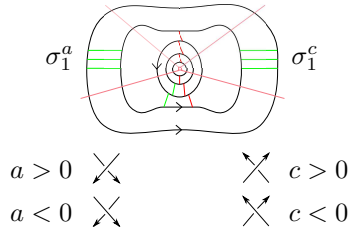
  
\labellist
 \pinlabel {$a<0$} at -38 13
 \pinlabel {$a>0$} at -38 55
 \pinlabel {$\sigma_1^a$} at -20 163
 \pinlabel {$c>0$} at 230 55
 \pinlabel {$c<0$} at 230 13
 \pinlabel {$\sigma_1^c$} at 217 163
\endlabellist
\begin{center}
\NTresNoTodasImparesTrenza
\end{center}
\caption{Braid for the pretzel $P(a,b,c)$ when at least one entry is even. A scar would have the opposite color if the corresponding entry were negative. Below: the crossings in the braid word corresponding to the extreme entries $a$ and $c$ have the signs of $a$ and $c$ respectively.} \label{FigureNTresNoTodasImparesTrenza}
\end{figure}

Now, we consider the case in which the three entries are odd, none of them equal to one. For convenience, orient the knot with the top segment running to the left (see Figure~\ref{FigureNTresTodasImparesOrientacionYMovimientosBasicos}, left side). Apply now $a'=(|a|-1)/2$ basic moves over the odd crossings of the first column, $b'=(|b|-1)/2$ basic moves over the even crossings of the central column and again $c'=(|c|-1)/2$ basic moves over the odd crossings of the third column. What we obtain (when seen in a sphere) are three groups of concentric Seifert circles, being those in the middle (a total of $2+b'$) compatible with those in the extremes (a total of $a'+c'$), as shown on the right side of Figure~\ref{FigureNTresTodasImparesOrientacionYMovimientosBasicos}.
\begin{figure}[ht!]
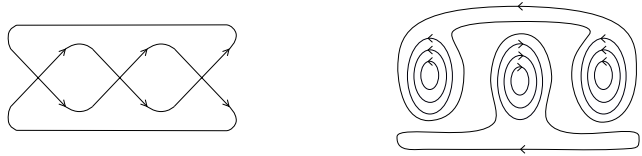
  
\labellist
\endlabellist
\begin{center}
\NTresTodasImparesOrientacion  
\hspace{2cm}
\NTresTodasImparesMovimientosBasicos
\end{center}
\caption{Orientation for the pretzel diagram $P(a,b,c)$ when $a$, $b$ and $c$ are odd, and the Seifert circles obtained after applying basic moves.} \label{FigureNTresTodasImparesOrientacionYMovimientosBasicos}
\end{figure}

In a second step, we use multiple reducing moves among the circles of both extremes, $a'$ on the left, $c'$ on the right. This produces a total of $2a'c'$ new crossings but the number of Seifert circles remains unchanged.  This provides a set of $a'+b'+c'+2$ Seifert circles with complexity zero, therefore a braid diagram with $a'+b'+c'+2 = \frac{|a|+|b|+|c|+1}{2}$ strands. The total number of crossings is $|a|+a'+|b|+b'+|c|+c'+2a'c' = |a|+|c|-1+ \frac{3|b|+|ac|}{2}$ and the obtained braid word is as stated (details in the general case of an odd number of entries will be given in the proof of Theorem~\ref{TheoremOddEntries}). If an extreme entry is equal to $1$, we do not need the second step. If $b=1$, we do not need to apply basic moves in the central column; in any case, the formula remains to be correct.
\end{proof}

\begin{theorem}\label{TheoremEvenEntries}
Let $P=P(a_1, a_2, \dots, a_n)$ be a pretzel diagram with an even number $n>2$ of entries. Then $P=\hat{\beta}$ where
$$
\begin{array}{rcl}
\beta & = & [1\ua (n-2)] \\
&&(n-1)^{a_1}(n-3)^{a_3}\dots 1^{a_{n-1}} \\
&&[-(n-2)\ua -1] \\
&&(n-1)^{a_2}(n-3)^{a_4}\dots 1^{a_n}.
\end{array}
$$
The braid $\beta$ has $2n+\sum_{i=1}^n |a_i| -4$ crossings and $n$ strands.
\end{theorem}

\begin{proof}
The pretzel diagram defines a knot if it has exactly one even entry. Otherwise we are dealing with a link. In any case we can orient the pretzel diagram with the top segment going to left, and in each column the two strands have the same orientation, as shown on the left of  Figure~\ref{FigureNparOrientacion}. The corresponding Seifert circles can be seen on the right of the same figure. 
\begin{figure}[ht!]
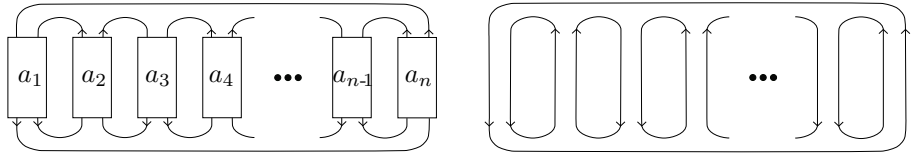
  
\labellist
\pinlabel {$a_1$} at 18 58
\pinlabel {$a_2$} at 68 58
\pinlabel {$a_3$} at 118 58
\pinlabel {$a_4$} at 168 58
\pinlabel {$a_{n\text{-}\!1}$} at 272 58
\pinlabel {$a_n$} at 323 58
\endlabellist
\begin{center}
\NparOrientacion
\hspace{0.5cm}
\NparCirculosSeifert
\end{center}
\caption{Chosen orientation if $n$ is even, and the corresponding Seifert circles.} \label{FigureNparOrientacion} 
\end{figure}

We then use $\frac{n}{2}-1$ column shifts, applied to the columns $3, 5, \dots, n-1$. Figure~\ref{FigureNparDesplazamientos} shows how to do it. 
\begin{figure}[ht!]
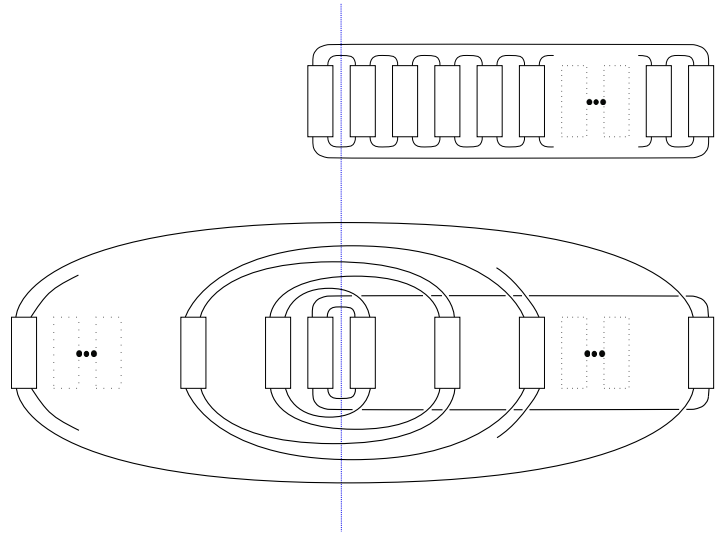
  
\labellist
\endlabellist
\begin{center}
\NparDesplazamientos
\end{center}
\caption{Column shifts of the columns $3, 5, \dots, n-1$, where $n$ is even.} \label{FigureNparDesplazamientos} 
\end{figure}

\begin{figure}[ht!]
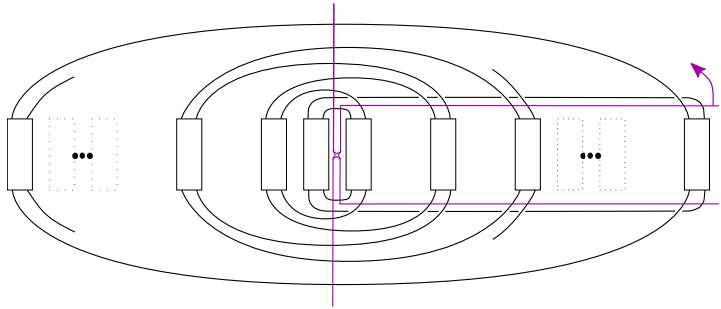
  
\labellist
\endlabellist
\begin{center}
\NparTrenza
\end{center}
\caption{The braid when $n$ is even, divided in four braid subwords.} \label{FigureNparTrenza} 
\end{figure}
The result are $n$ columns with entries, from left to right, $a_{n-1}$, $a_{n-3}$, $\dots$, $a_3$, $a_1$, $a_2$, $a_4$, $\dots$, $a_{n-2}$, $a_n$, and $4$ extra crossings for each displaced column, two with the top horizontal segment, two with the bottom horizontal segment, therefore a total of $\sum_{i=1}^n |a_i| + 4(\frac{n}{2}-1) = 2n+\sum_{i=1}^n |a_i| -4$ crossings. The configuration of the Seifert circles has then complexity equal to zero, and in the braid word the two groups of crossings with both horizontal segments provide the braid words $[1\ua(n-2)]$ and its inverse. There are also two other groups (see Figure~\ref{FigureNparTrenza}), that correspond to the second and four lines in the statement. The number of strands is $n$.
\end{proof}

\newpage

\begin{theorem} \label{TheoremOddEntries}
Let $P=P(a_1, a_2, \dots, a_n)$ be a pretzel diagram with an odd number $n$ of entries. Then $P=\hat{\beta}$ where
\begin{enumerate}
\item If there is at least one even entry, that we may assume to be $a_1$ by the cyclic property, then
$$
\begin{array}{rcl}
\beta & = & 1^{a_{n-1}} 3^{a_{n-3}} \dots (n-2)^{a_2} \\
&&s_{a_1} [1-n \da 2-n-\frac{|a_1|}{2}] \\
&&[2-n \ua -1] \\
&&1^{a_n}3^{a_{n-2}}\dots (n-2)^{a_3} \\
&& [1\ua n-2] \\
&& s_{a_1}( [n-1 \ua n-3+\frac{|a_1|}{2}] \, [2-n-\frac{|a_1|}{2} \ua 1-n]).
\end{array}
$$

\item If there is no even entries at all, 
$$
\beta = G(n) \dots G(3) G(2) \, 
T(1) G(2)^{-1}T(2) G(3)^{-1} \dots T(n-1)G(n)^{-1} T(n) 
$$
where, for $j=1, \dots, n$ 
$$
\begin{array}{rcl}
T(j)  
& = & \epsilon_j ([-2\da -b_j-1] [1\ua b_j] [-b_j-1 \ua -1]) \\
& = & \sigma_{2}^{-\epsilon_j}\sigma_{3}^{-\epsilon_j}\dots\sigma_{b_j+1}^{-\epsilon_j} 
(\sigma_{1}^{\epsilon_j}\sigma_{2}^{\epsilon_j}\dots\sigma_{b_j}^{\epsilon_j})
\sigma_{b_j+1}^{-\epsilon_j}\dots\sigma_{2}^{-\epsilon_j}\sigma_{1}^{-\epsilon_j} \\
\end{array}
$$
being $b_j = \frac{|a_j|-1}{2}$ and $\epsilon_j = \frac{a_j}{|a_j|}$ the sign of the entry $a_j$, and for $k= 2, \dots, n$
$$
G(k)=G(k-1,k)\dots G(2,k)G(1,k),
$$
where for $1\leq i < j \leq n$
$$
G(i,j) = [-B(i,j)-b_j \da -B(i,j)-b_j-b_i+1]^{\ua b_j \ua}
$$
being $B(i,j) = 2+b_{i+1}+\dots + b_{j-1}$, $1\leq i < j \leq n$.
\end{enumerate}
In the first case $\beta$ has $\frac{|a_1|}{2}+\sum_{j=1}^n |a_j| + 2n - 5$ crossings and $\frac{|a_1|}{2}+n-1$ strands. In the second case $\beta$ has 
$\frac{1}{2}\left(3 \sum_{j=1}^n |a_j| + \sum_{1\leq i < j \leq n} (|a_i|-1)(|a_j|-1) - n \right)$ crossings and $\frac{1}{2} \left(4-n+\sum_{j=1}^n |a_j|\right)$ strands. 
\end{theorem}

\begin{proof}
Assume first that there is at least one even entry. By the cyclic property we can assume that $a_1$ is even. We say that a column is N-oriented (S-oriented) if its two strands are oriented from bottom to top (from top to bottom); otherwise we say that the column is not well oriented. We claim that we can orient the pretzel diagram in such a way that the columns in even position are N-oriented and those in odd positions are S-oriented, except the first one, that is not well oriented (see Figure~\ref{FigureNimparPrimeraColumnaParOrientacion}).

\begin{figure}[ht!]
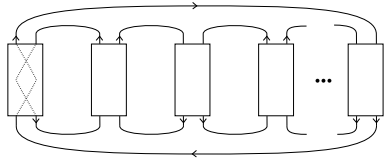
  
\labellist
\endlabellist
\begin{center}
\NimparPrimeraColumnaParOrientacion
\end{center}
\caption{How to orient the pretzel diagram when $n$ is odd and $a_1$ is even.} \label{FigureNimparPrimeraColumnaParOrientacion} 
\end{figure}

To prove that such an orientation is possible we proceed by induction on the number $k$ of even entries. The case $k=1$ is left to the reader as an exercise. Assume now that there are other even entries, and let $j$ be the first even entry after the first one. Compare now the pretzel diagrams $P=P(a_1, \dots, a_n)$ and  $P'=P(a_1, \dots, a_{j-1}, 1 + a_j, a_{j+1}, \dots, a_n)$  where $P'$ has an extra crossing in the column $j$. By induction we may assume the wished orientation in $P'$. Now this orientation can be carried to $P$ by choosing the appropriate orientation of the extra component of $P$ situated between the columns $1$ and $j$ (see Figure~\ref{FigureNimparPrimeraColumnaParOrientando}).
\begin{figure}[ht!]
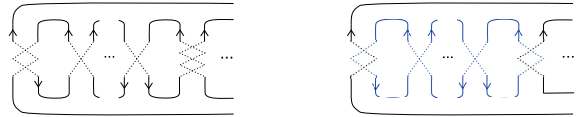
  
\labellist
\endlabellist
\begin{center}
\NimparPrimeraColumnaParOrientando
\end{center}
\caption{Carrying the orientation from $P'$ to $P$.} \label{FigureNimparPrimeraColumnaParOrientando} 
\end{figure}

We now apply $\frac{|a_1|}{2}$ basic moves on the even crossings of the fist column, obtaining $\frac{|a_1|}{2}$ concentric Seifert circles and $\frac{|a_1|}{2}-1$ extra crossings, as shown in Figure~\ref{FigureNimparPrimeraColumnaParMovimientosBasicos}. The color of the scars is the opposite (red are green and vice versa) if $a_1<0$.
\begin{figure}[ht!]
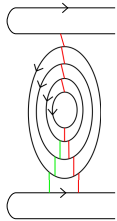
  
\labellist
\endlabellist
\begin{center}
\NimparPrimeraColumnaParMovimientosBasicos
\end{center}
\caption{Basic moves in the first column of $P(a_1, \dots, a_n)$, $n$ odd and $a_1$ even. Color of each scar is the opposite if $a_1<0$. Here $a_1=8$.} \label{FigureNimparPrimeraColumnaParMovimientosBasicos} 
\end{figure}
  
We now apply column shifts to the columns $3, 5, \dots, n$, as shown in Figure~\ref{FigureNimparPrimeraColumnaParDesplazamientos}. 
\begin{figure}[ht!]
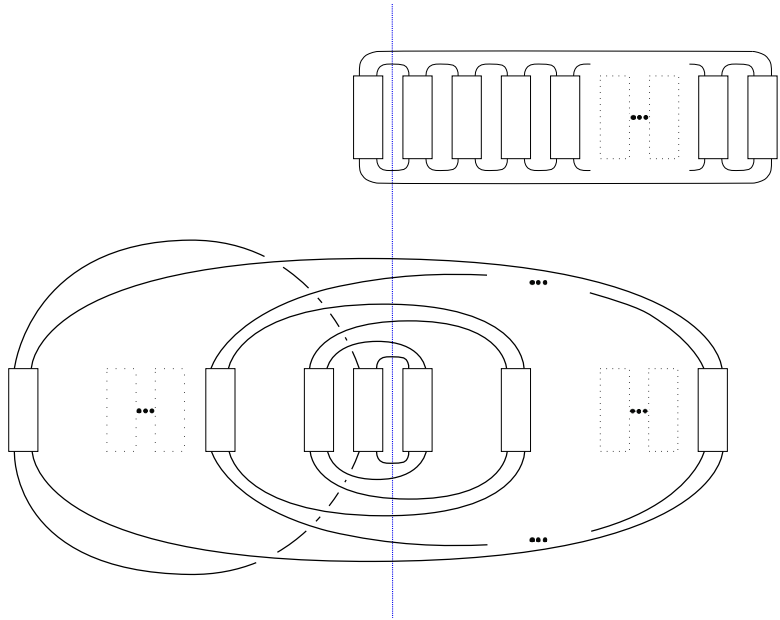
  
\labellist
\endlabellist
\begin{center}
\NimparPrimeraColumnaParDesplazamientos
\end{center}
\caption{Column shifts of the columns $3, 5, \dots, n$, where $n$ is odd and $a_1$ even.} \label{FigureNimparPrimeraColumnaParDesplazamientos} 
\end{figure}

\begin{figure}[ht!]
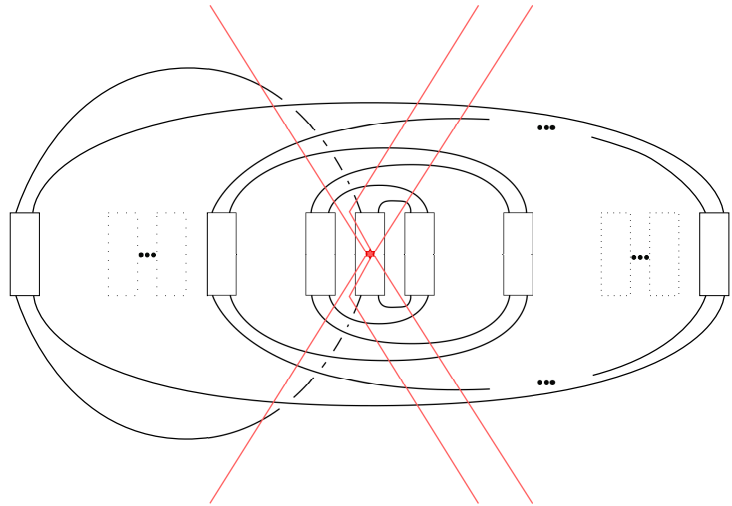
  
\labellist
\endlabellist
\begin{center}
\NimparPrimeraColumnaParTrenza
\end{center}
\caption{The braid when $n$ is odd and $a_1$ is even, divided in six braid words.} \label{FigureNimparPrimeraColumnaParTrenza} 
\end{figure}

\newpage

We obtain a total of $N = \frac{|a_1|}{2} + 2\frac{n-1}{2} = \frac{|a_1|}{2}+n-1$ strands for the final braid, divided in six braid subwords as shown in Figure \ref{FigureNimparPrimeraColumnaParTrenza}. These six subwords are the following:
$$
\begin{array}{rcl}
\beta & = & 1^{a_{n-1}} 3^{a_{n-3}} \dots (n-2)^{a_2} \\
&&s_{a_1} [\frac{|a_1|}{2}-N \ua -(N-1)]\\
&&[-\left( 2\frac{n-1}{2} - 1 \right) \ua -1]\\
&&1^{a_n}3^{a_{n-2}}\dots (n-2)^{a_3} \\
&& [1\ua 2\frac{n-1}{2}-1]\\
&& s_{a_1}\left([(N-2)- (\frac{|a_1|}{2}-2) \ua N-2] [-(N-1) \ua -(N-1)+ (\frac{|a_1|}{2}-1)]\right) \\
& = & 1^{a_{n-1}} 3^{a_{n-3}} \dots (n-2)^{a_2} \\
&& s_{a_1} [1-n \ua 2-n-\frac{|a_1|}{2}] \\
&& [2-n \ua -1]  \\
&& 1^{a_n}3^{a_{n-2}}\dots (n-2)^{a_3} \\
&& [1\ua n-2] \\
&& s_{a_1} \left([n-1 \ua n-3+\frac{|a_1|}{2}] \, [2-n-\frac{|a_1|}{2} \ua 1-n]\right)
\end{array}
$$
where, for the third line, it must be noted that the shift of the last column only gives one crossing on the top. This ends the proof when there is at least one even entry.

We concentrate now on the second case, when $n$ and any entry $a_i$ is odd. We have then a knot which can be oriented as shown in Figure~\ref{FigureNimparTodaEntradaImparOrientacion}.
\begin{figure}[ht!]
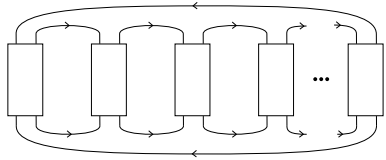
  
\labellist
\endlabellist
\begin{center}
\NimparTodaEntradaImparOrientacion
\end{center}
\caption{Orientation when $n$ and any entry is odd.} \label{FigureNimparTodaEntradaImparOrientacion} 
\end{figure}

We start by applying $b_j=\frac{|a_j|-1}{2}$ basic moves on crossings in odd position (first, third and so on until the antepenultimate) of the column $j$, for $j=1, \dots, n$. In the column $j$ we obtain $b_j$ concentric counterclock oriented Seifert circles and $b_j$ extra crossings. These circles and the corresponding scars (of the crossings)  are shown in Figure~\ref{FigureNimparTodaEntradaImparMovimientosBasicos} for the pretzel diagram $P(1,3,5,7,9)$. Note that the colour of the scars in the column $j$ should be interchanged if it were $a_j<0$.
\begin{figure}[ht!]
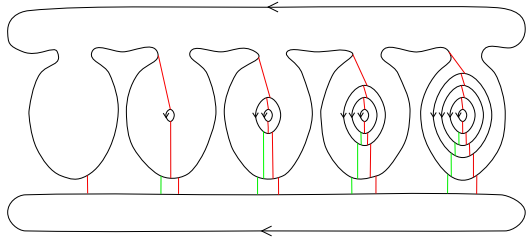
  
\labellist
\endlabellist
\begin{center}
\NimparTodaEntradaImparMovimientosBasicos
\end{center}
\caption{Result of applying $b_j=\frac{|a_j|-1}{2}$ basic moves to the odd crossings in each $j$th-column, when $n$ and any $a_j$ is odd.} 
\label{FigureNimparTodaEntradaImparMovimientosBasicos} 
\end{figure}

Circles in different columns are not compatible, so we apply multiple reducing moves among all the circles in different columns. The blue box in Figure~\ref{FigureNimparTodaEntradaImparEspacioParaReduccion} bounds the space where these moves are performed. What we obtain is shown in Figure~\ref{FigureNimparTodaEntradaImparReduccionMultiple}, where a strand labeled with a non-negative integer $k$ means a collection of $k$ strands that run parallel to the drawn strand, hence 
$\stackrel{k}{\line(1,0){30}}$
means 
$\put(0,0){\line(1,0){30}}
\put(10,3){\dots}
\put(0,7){\line(1,0){30}}
\hspace{1cm} \,$
($k$ strands).
Note also that in Figure~\ref{FigureNimparTodaEntradaImparReduccionMultiple} a strand coming from the column $j$ goes over the strands coming from the column $i$ if $i<j$, hence the crossings of the top half space are negative, those of the lower half space are positive. In the space for multiple reducing moves we find a total of $2\sum_{1\leq i<j\leq n} b_ib_j = \frac{1}{2} \sum_{1\leq i<j\leq n} (|a_i|-1)(|a_j|-1)$ crossings. 
\begin{figure}[ht!]
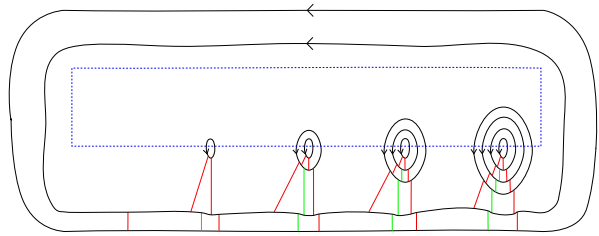
  
\labellist
\endlabellist
\begin{center}
\NimparTodaEntradaImparEspacioParaReduccion
\end{center}
\caption{Space for multiple reducing moves.} \label{FigureNimparTodaEntradaImparEspacioParaReduccion} 
\end{figure}

\begin{figure}[ht!]
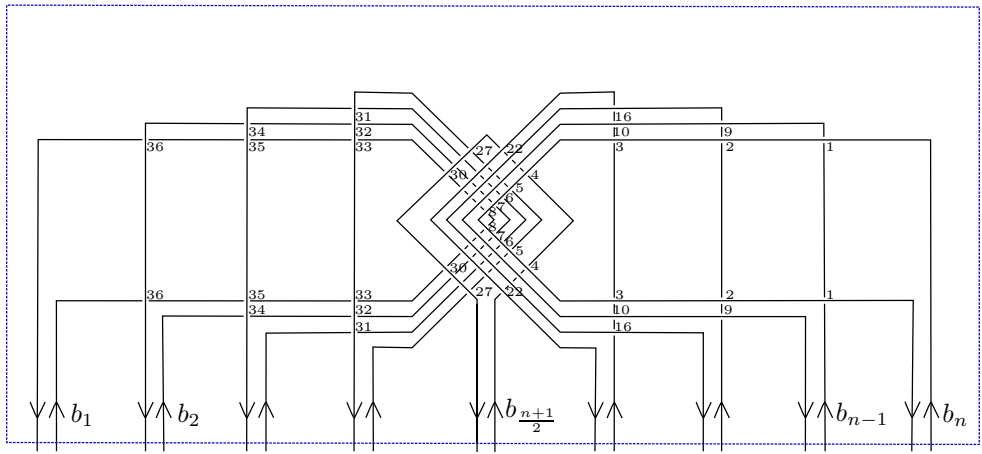
  
\labellist
\pinlabel {$b_1$} at 45 22
\pinlabel {$b_2$} at 108 22
\pinlabel {$b_{\frac{n+1}{2}}$} at 310 22
\pinlabel {$b_{n-1}$} at 505 22
\pinlabel {$b_n$} at 560 22
\pinlabel {\tiny $1$} at 486 180
\pinlabel {\tiny $2$} at 427 180
\pinlabel {\tiny $3$} at 362 180
\pinlabel {\tiny $4$} at 312 164 
\pinlabel {\tiny $5$} at 303 156
\pinlabel {\tiny $6$} at 297 150
\pinlabel {\tiny $7$} at 292 145
\pinlabel {\tiny $8$} at 287 142
\pinlabel {\tiny $9$} at 426 189
\pinlabel {\tiny $10$} at 363 189
\pinlabel {\tiny $16$} at 364 198
\pinlabel {\tiny $22$} at 300 179
\pinlabel {\tiny $27$} at 282 178
\pinlabel {\tiny $30$} at 267 164
\pinlabel {\tiny $31$} at 211 198
\pinlabel {\tiny $32$} at 211 189
\pinlabel {\tiny $33$} at 211 180
\pinlabel {\tiny $34$} at 148 189
\pinlabel {\tiny $35$} at 148 180
\pinlabel {\tiny $36$} at  88 180
\pinlabel {\tiny $1$} at 486 93
\pinlabel {\tiny $2$} at 427 93
\pinlabel {\tiny $3$} at 362 93
\pinlabel {\tiny $4$} at 312 110 
\pinlabel {\tiny $5$} at 303 119
\pinlabel {\tiny $6$} at 297 124
\pinlabel {\tiny $7$} at 292 128
\pinlabel {\tiny $8$} at 287 133
\pinlabel {\tiny $9$} at 426 84
\pinlabel {\tiny $10$} at 363 84
\pinlabel {\tiny $16$} at 364 74
\pinlabel {\tiny $22$} at 300 95
\pinlabel {\tiny $27$} at 282 95
\pinlabel {\tiny $30$} at 267 109
\pinlabel {\tiny $31$} at 211 74
\pinlabel {\tiny $32$} at 211 84
\pinlabel {\tiny $33$} at 211 93
\pinlabel {\tiny $34$} at 148 84
\pinlabel {\tiny $35$} at 148 93
\pinlabel {\tiny $36$} at  88 93
\endlabellist
\begin{center}
\NimparTodaEntradaImparReduccionMultiple
\end{center}
\caption{Multiple reducing moves.} \label{FigureNimparTodaEntradaImparReduccionMultiple} 
\end{figure}

After performing basic moves in each column and then multiple reducing moves among arcs of different columns, we have obtained a presentation of the pretzel knot as a closed braid. The set of crossings in the braid originated by the crossing of the ($b_i$ parallel strands represented by the) strand $i$ (that coming from the column $i$) with the strand $j$, for $i<j$, can be ordered as a subword of the braid word, and it will be called $G(i,j)$. In particular, $G(i,j)$ is the empty braid word if $b_i=0$ or $b_j=0$. Moreover, the braid word can be divided globally into $3n-2$ subwords as we show in Figure~\ref{FigureNimparTodaEntradaImparGrupos}. The wanted braid $\beta$ can be then written as
$$
\beta = G(n) \dots G(3) G(2) \, 
T(1) G(2)^{-1}T(2) G(3)^{-1} \dots T(n-1)G(n)^{-1} T(n)
$$
where $T(j)$ takes into account the crossings derived from the basic moves in the column $j$, and, following the numeration in Figure~\ref{FigureNimparTodaEntradaImparReduccionMultiple}, 
$$
G(k)=G(k-1,k)\dots G(2,k)G(1,k).
$$
\begin{figure}[ht!]
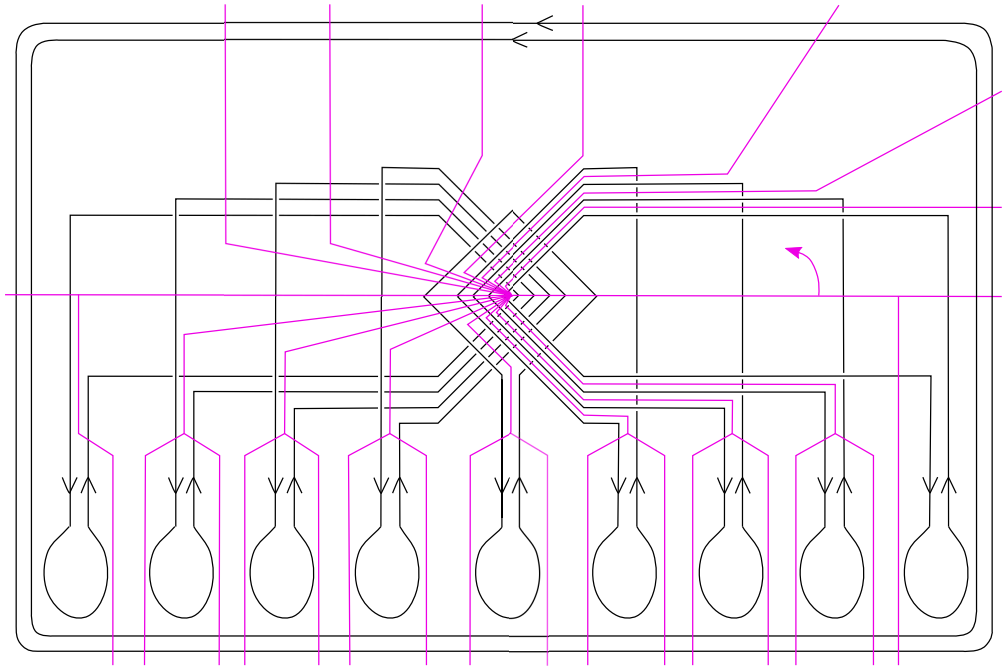
  
\labellist
\endlabellist
\begin{center}
\NimparTodaEntradaImparGrupos
\end{center}
\caption{Mixed diagram representing the final braid: the top part is made of {\it non-amplified} crossings; the lower part shows partial Seifert circles (here $n=9$). This picture allows to distinguish the different subwords in the final braid word. The oriented arc indicates where we start to write the braid word.} \label{FigureNimparTodaEntradaImparGrupos} 
\end{figure}

The braid word $G(i,j)$ is (see Figure~\ref{FigureNimparTodaEntradaImparReduccionMultipleDetalle})
$$
\begin{array}{rcl}
G(i,j) & = & 
(\sigma_{B(i,j)+b_j}^{-1} \sigma_{B(i,j)+b_j+1}^{-1} \dots \sigma_{B(i,j)+b_j+b_i-1}^{-1}) \\
&& (\sigma_{B(i,j)+b_j-1}^{-1}\sigma_{B(i,j)+b_j}^{-1} \dots \sigma_{B(i,j)+b_j-1+b_i-1}^{-1})\\
&& \dots (\sigma_{B(i,j)+1}^{-1}\sigma_{B(i,j)+2}^{-1} \dots \sigma_{B(i,j)+b_i}^{-1}) \\
& = & (-B(i,j)-b_j\da -B(i,j)-b_j-b_i+1) \\
&& (-B(i,j)-b_j+1\da -B(i,j)-b_j+1-b_i+1) \\
&& \dots (-B(i,j)-1\da -B(i,j)-b_i) \\
& = & [-B(i,j)-b_j \da -B(i,j)-b_j-b_i+1]^{\ua b_j \ua}
\end{array}
$$
where $B(i,j)=2+b_{i+1}+\dots + b_{j-1}$ is the number of strands that surround the outside of the crossing. Here it must be understood from the notation that $G(i,j)$ is the empty word if $b_i=0$, that is, if $a_i=\pm 1$. But by definition $G(i,j)$ is also the empty word if $b_j=0$. Note also that one could read $G(i,j)$ following the order of the crossings $1,4,7,10;2,5,8,11;3,6,9,12$ in Figure~\ref{FigureNimparTodaEntradaImparReduccionMultipleDetalle}.

\begin{figure}[ht!]
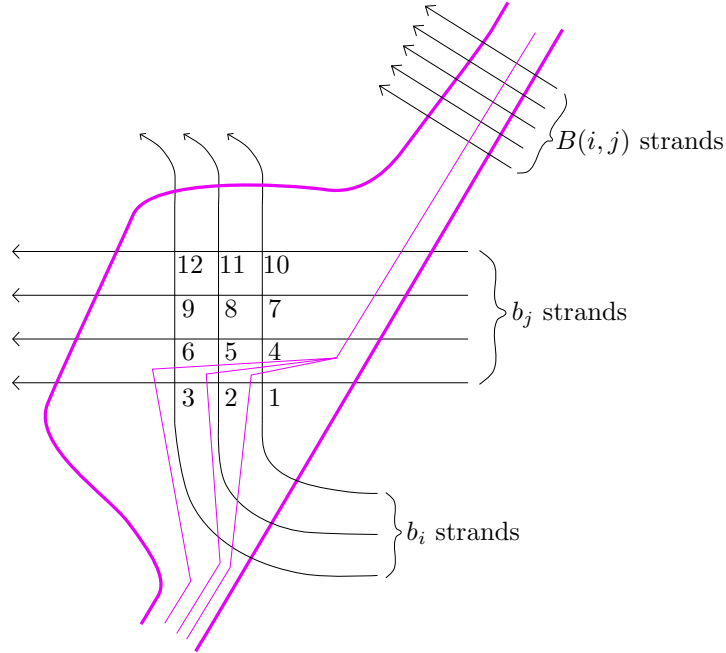
  
\labellist
\pinlabel {$1$} at 198 192
\pinlabel {$2$} at 165 192
\pinlabel {$3$} at 133 192
\pinlabel {$4$} at 198 227
\pinlabel {$5$} at 165 227
\pinlabel {$6$} at 133 227
\pinlabel {$7$} at 198 259
\pinlabel {$8$} at 165 259
\pinlabel {$9$} at 133 259
\pinlabel {$10$} at 199 293
\pinlabel {$11$} at 166 293
\pinlabel {$12$} at 134 293
\pinlabel {$b_i$ strands} at 340 90
\pinlabel {$b_j$ strands} at 420 255
\pinlabel {$B(i,j)$ strands} at 473 385
\endlabellist
\begin{center}
\NimparTodaEntradaImparReduccionMultipleDetalle
\end{center}
\caption{Detail of the multiple reducing move: amplified crossing of Figure~\ref{FigureNimparTodaEntradaImparReduccionMultiple}. 
} \label{FigureNimparTodaEntradaImparReduccionMultipleDetalle} 
\end{figure}

The subword $T(j)$ is 
$$
\begin{array}{rcl}
T(j)  
& = & (-2\da -b_j-1) (1\ua b_j) (-b_j-1 \ua -1) \\
& = & \sigma_{2}^{-1}\sigma_{3}^{-1}\dots\sigma_{b_j+1}^{-1} 
(\sigma_{1}\sigma_{2}\dots\sigma_{b_j})
\sigma_{b_j+1}^{-1}\dots\sigma_{2}^{-1}\sigma_{1}^{-1}
\end{array}
$$
where the central parenthesis corresponds to the green crossings as can be seen in any column in 
Figure~\ref{FigureNimparTodaEntradaImparEspacioParaReduccion} (colors must be interchanged if the corresponding $a_j<0$). Note that $T(j)=\sigma_2^{-1} \sigma_1 \sigma_2^{-1} \sigma_1^{-1}$ if $b_j=1$, that is, $a_j=\pm 3$, and $T(j)=\sigma_1^{-1}$ if $b_j=0$, that is, $a_j=\pm 1$.  In general, in $T(j)$ there are $3b_j+1$ crossings.

We finally count the number of crossings and strands in the obtained braid. The number $c$ of crossings is 
$$
\begin{array}{rcl} 
c & = & \sum_{j=1}^n (3b_j+1) + 2 \sum_{1\leq i < j \leq n} b_ib_j \\ 
&&\\
& = & n + 3 \sum_{j=1}^n b_j + 2 \sum_{1\leq i < j \leq n} b_ib_j \\
&&\\
& = & n + 3 \sum_{j=1}^n \frac{|a_j|-1}{2} + 2 \sum_{1\leq i < j \leq n} \frac{|a_i|-1}{2}\frac{|a_j|-1}{2} \\
&&\\
& = & -\frac{n}{2} + \frac{3}{2} \sum_{j=1}^n |a_j| + \frac{1}{2} \sum_{1\leq i < j \leq n} (|a_i|-1)(|a_j|-1)
\end{array}
$$
and the total number $b$ of strands is 
$$
b = 2+\sum_{j=1}^n b_j
= 2+\sum_{j=1}^n \frac{|a_j|-1}{2}
= 2+ \frac{1}{2} (\sum_{j=1}^n |a_j| - n)
=\frac{1}{2} \left(4-n+\sum_{j=1}^n |a_j|\right).
$$
\end{proof}

\begin{figure}[ht!]
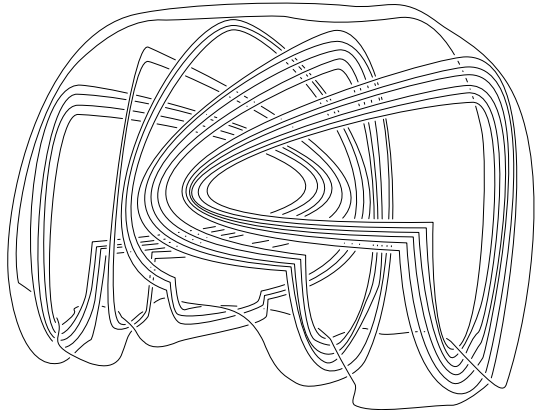
  
\begin{center}
\Ejemplo
\end{center}
\caption{The pretzel knot $P(9, 5, 7, 11, 13)$ as closed braid.}
\label{FigureEjemplo} 
\end{figure}
As an example, we show in Figure~\ref{FigureEjemplo} the braid ($22$ strands and $375$ crossings) for the pretzel knot $P(9, 5, 7, 11, 13)$, according to the second point of Theorem~\ref{TheoremOddEntries}. You can quickly find a braid whose closure is your favourite pretzel link in the following web site, in which the first author has implemented the algorithms described in this paper:
\url{https://adelpozoman.es/blog/braids-for-pretzel/}


\begin{center}
\begin{tabular}{c}
\'Angel del Pozo Manglano\\
ETSIDI, Universidad Polit\'ecnica de Madrid \\
 {\it a.delpozo@alumnos.upm.es} 
\end{tabular}

\begin{tabular}{c}
Pedro M. Gonz\'alez Manch\'on (corresponding author) \\
Department of Applied Mathematics to Industrial Engineering \\
ETSIDI, Universidad Polit\'ecnica de Madrid \\
{\it pedro.gmanchon@upm.es} \\
\end{tabular}
\end{center}

\end{document}